\documentclass[a4paper,11pt]{article}
\usepackage[T1]{fontenc}
\usepackage[latin1]{inputenc}
\usepackage{amsmath,amsfonts}
\usepackage{amsthm}
\usepackage{amssymb}
\usepackage{authblk}
\usepackage[USenglish]{babel}
\usepackage{enumerate}
\usepackage{cite}
\usepackage[left=3cm, right=3cm, top=3cm, bottom=3cm]{geometry}
\usepackage{xcolor}
\usepackage{todonotes}
\usepackage[final, pdftex, pdfpagelabels, pdfstartview = {FitH}, bookmarks, plainpages = false, linktoc=all,linkcolor=black, citecolor=blue, urlcolor=blue,
filecolor=black]{hyperref}
\usepackage{cleveref}
\usepackage{caption}
\usepackage{subcaption}
\usepackage{dynkin-diagrams}
\usepackage{titlesec}

\theoremstyle{plain}
\newtheorem{theorem}{Theorem}[section]
\newtheorem{lemma}[theorem]{Lemma}
\newtheorem{prop}[theorem]{Proposition}
\newtheorem{cor}[theorem]{Corollary}

\newtheorem{conjecture}[theorem]{Conjecture}
\newtheorem{lettertheorem}{Theorem}

\theoremstyle{definition}
\newtheorem{definition}[theorem]{Definition}

\newtheorem{remark}[theorem]{Remark}
\newtheorem{example}[theorem]{Example}

\theoremstyle{remark}
\newtheorem{claim}[theorem]{Claim}

\titleformat{\subsection}[runin]
       {\normalfont\bfseries}
       {\thesubsection}
       {0.5em}
       {}
       [.]

\renewcommand{\tilde}{\widetilde}
\renewcommand{\bar}{\overline}

\newcommand{\bbQ}{\mathbb{Q}}

\newcommand{\bbZ}{\mathbb{Z}}

\newcommand{\FA}{A_{\text{id}}}

\DeclareMathOperator{\spa}{span}

\DeclareMathOperator{\cone}{cone}
\DeclareMathOperator{\conv}{Conv}
\DeclareMathOperator{\vol}{Vol}
\DeclareMathOperator{\relvol}{relVol}

\numberwithin{equation}{section}

\title{On the size of Bruhat intervals } 

\author{ Federico Castillo\thanks{Pontificia Universidad Catolica, Santiago, Chile}, Damian de la Fuente\thanks{LAMFA, Universit\'e de Picardie Jules Verne, Amiens, France}, Nicolas Libedinsky\thanks{Universidad de Chile, Santiago, Chile}  and David Plaza\thanks{Universidad de Talca, Talca, Chile}  }

\date{ }

\begin{document}
\maketitle

\begin{abstract}
For affine Weyl groups and elements associated to dominant coweights, we present a  convex geometry formula for the size of the corresponding lower Bruhat intervals. Extensive computer calculations for these groups have led us to believe that a similar formula exists for all lower Bruhat intervals.

\end{abstract}

\maketitle

\section{Introduction}
\subsection{Generalities}\label{gen} 



While calculating with indecomposable Soergel bimodules \cite{LP} and Kazhdan-Lusztig polynomials \cite{batistelli2023kazhdan},  \cite{PreCanonical} for affine Weyl groups, it became apparent that finding formulas for the cardinalities of lower Bruhat intervals played a crucial role. 

Surprisingly, little is known apart from length $2$ (general) intervals \cite[Lemma 2.7.3]{BB}, lower intervals for smooth elements in Weyl groups \cite{OhPostnikovYoo}, \cite{MSY} and related results for affine Weyl groups \cite{RS},  \cite{BE}. 
Although the known cases are particularly important from a geometric viewpoint, to the best of the authors' knowledge there is no program for finding cardinalities of all lower intervals.
This motivated us to embark on a series of papers aimed at bridging this gap in the case of affine Weyl groups.
In this paper, we relate the Bruhat order with convex geometry.
With ideas similar to those presented here, we expect to go beyond lower intervals into the realm of general intervals in future work. 

In this paper we study, for any affine Weyl group,  the lower interval for the element  $\theta(\lambda)$ (see Definition \ref{def: theta}) associated to a dominant coweight $\lambda$.
These are the elements that originally appeared in the research on Soergel bimodules and Kazhdan-Lusztig polynomials mentioned above; they are intimately related to representation theory (character formulas for Lie groups, geometric Satake equivalence, quantum groups, among others).

The main result of this paper is a formula relating the cardinality of the lower interval $[\mathrm{id}, \theta(\lambda)]$ and the volumes of the faces of a certain polytope.
We guessed this formula by examining the $\widetilde{A}_2$ case and
assuming that certain phenomena that hold there continue to hold in higher dimensions. Although these phenomena turned out to be low-rank accidents, the formula miraculously survived.


This paper makes apparent that Bruhat intervals for affine Weyl groups are intricately connected to Euclidean geometry. 
Another manifestation of this connection is the observation in the work in progress \cite{BLV}, that for an affine Weyl group possibly all ``non-silly'' isomorphisms of  Bruhat intervals (of length bigger than the order of the finite Weyl group) are just Euclidean translations of the connected components of the interval.
A different kind of connection between these two worlds, but this time for the symmetric group, was developed in \cite{kodama2015full} and further explored in \cite{tsukerman2015bruhat}: the \emph{Bruhat interval polytopes} consisting of the convex hull of permutations in an arbitrary Bruhat interval.

\subsection{The  \texorpdfstring{$\widetilde{A}_2$}{} case}\label{a2}
Let us consider $W$ the affine Weyl group of type $\widetilde{A}_2$, and the usual identification between elements in $W$ and triangles (alcoves) in the tessellation of the plane by equilateral triangles. If $x$ is an element of $W$, when we write $x\subset \mathbb{R}^2$, we mean the set of points in the closure of the alcove corresponding to $x$ (the closed triangle).  
In Figure 1 we have the simple roots $\alpha_1$ and $\alpha_2$ in blue and in red, and the fundamental weights $\varpi_1$ and $\varpi_2$. 
For a dominant weight $\lambda \in X^+ 
:=\mathbb{Z}_{\geq0}\varpi_1+\mathbb{Z}_{\geq0}\varpi_2$ (depicted by a white dot in
Figure 1),
let $\theta(\lambda)\in W$ denote the $\lambda$-translate of
the opposite of the fundamental alcove: those are the grey
triangles.

Let also $Y_\lambda$  denote $\mathrm{Conv}(W_f\cdot
\lambda)$, the convex hull of the orbit of $\lambda$ under the finite
Weyl group $W_f$. For $\lambda = \varpi_1+2\varpi_2$, it is the yellow
hexagon in Figure 2. The faces of $Y_\lambda$ containing $\lambda$ are
$$F_J := Y_\lambda\cap ( \lambda + \sum_{i \in J} \mathbb{R}\alpha_i) \quad J
\subset \{1,2\}.$$
\vspace{0.5cm}

\begin{figure}[ht]
    \centering
    \begin{minipage}{0.45\textwidth}
        \centering
        \includegraphics[width=0.9\textwidth]{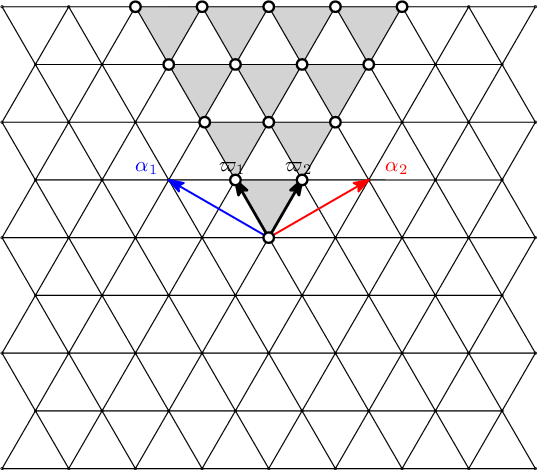} 
        \caption{}
        \label{fig:tetha}
    \end{minipage}\hfill
    \begin{minipage}{0.45\textwidth}
        \centering
        \includegraphics[width=0.9\textwidth]{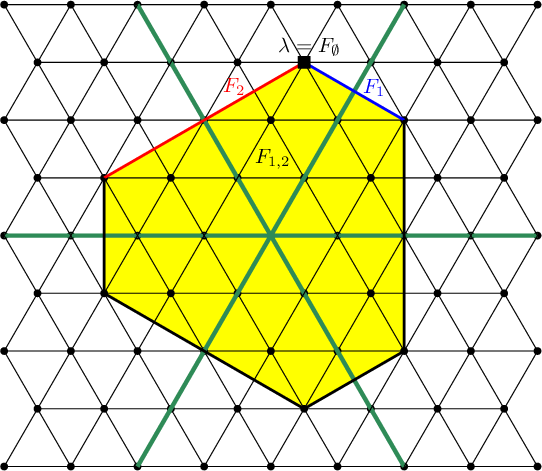} 
        \caption{}
    \end{minipage}
\end{figure}
For $x\in W$ we will denote $ \leq x:=\{w\in W\, \vert \, w\leq x \}$.
In Figure 3  we draw the set $\leq \theta(\lambda)$ (with $\lambda$ as before).
It is the union of all the colored sets. 

\vspace{0.5cm}

\begin{figure}[ht]
    \centering
    \begin{minipage}{0.45\textwidth}
        \centering
        \includegraphics[width=0.9\textwidth]{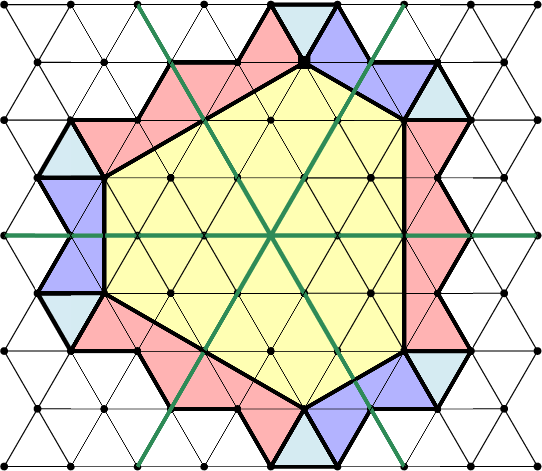} 
        \caption{}
    \end{minipage}\hfill
    \begin{minipage}{0.45\textwidth}
        \centering
        \includegraphics[width=0.9\textwidth]{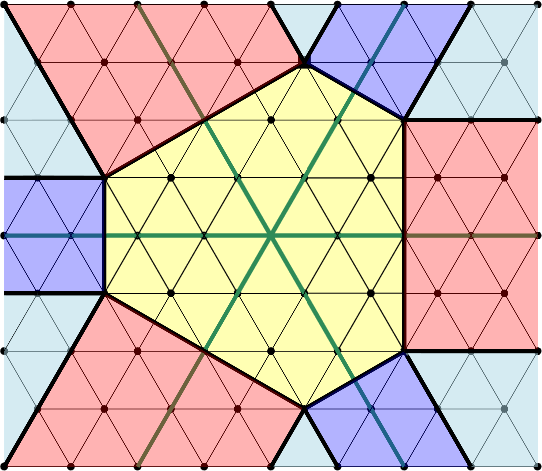} 
        \caption{}
    \end{minipage}
\end{figure}

\vspace{0.5cm}

Let us suppose only in this introduction, to simplify the formulas, that the volume of one alcove is $1$, so that the volume of $\leq \theta(\lambda)\subset \mathbb{R}^2$ is equal to the cardinality of $\leq \theta(\lambda)\subseteq W$.
Our initial observation was that there are four real numbers $\mu_{J}$ (independent of $\lambda$) with $J\subseteq \{1,2\}$ such that 
\begin{equation}\label{babypick}
    \vert \leq \theta(\lambda) \vert = 
\mu_{1,2}\mathrm{Area}(F_{1,2})+\mu_{1}\mathrm{Length}(F_{1})+\mu_{2}\mathrm{Length}(F_{2})+\mu_{\emptyset}\mathrm{Card}(F_{\emptyset}).\end{equation}


\begin{remark}
   The reader may notice that the formula presented here bears strong similarities to Pick's theorem (see \eqref{eq: pick}). For the proof of Theorem \ref{geofor}, a generalization of the formula \ref{babypick} applicable to any root system, we use a generalized version of Pick's theorem developed by Berline and Vergne. For more details see Example \ref{Pick}.
\end{remark}

In Figure 4 there is a  partition of the plane into 13 parts, one for each face of $Y_\lambda$, such that when intersected with $\leq \theta(\lambda)$, one obtains Figure 3.
That division can be done for any convex polytope and is called the set of normal cones. For $C$ a face of $Y_\lambda$ we call $\mathrm{Nor}(C)$ the corresponding region. 

The low-rank accidents behind the formula, mentioned in Section \ref{gen}, are the following. First, that   $$Y_\lambda\  \subseteq \ \ \leq \theta(\lambda)\   \subseteq \  \mathbb{R}^2.$$
Second, that in Figure 3 the number $\mu_{1,2}\mathrm{Area}(F_{1,2})$
is the volume of the yellow part,  
$\mu_{1}\mathrm{Length}(F_{1})$ is the volume of the blue part,  $\mu_{2}\mathrm{Length}(F_{2})$ is the volume of the red part and finally that $\mu_{\emptyset}\mathrm{Card}(F_{\emptyset})$
is the volume of the light blue part. In general, it is that $\mu_{J}\mathrm{Vol}(F_{J})$ is the volume of $W_f\cdot (\mathrm{Nor}(F_J)\, \, \cap\,   \leq \theta(\lambda))$.

Although these low-rank accidents are correct for $\widetilde{A}_2$ and $\widetilde{A}_3$, they fail in higher ranks.
The first one is not true already for $\widetilde{A}_4$.
The second one fails in $\widetilde{A}_{24},$ because there is a $\mu_J<0,$ (see Remark \ref{remarkfinal}) so $\mu_{J}\mathrm{Vol}(F_{J})$ can not be  the volume of some set.

\subsection{Results}
For any root system $\Phi$ one has an associated affine Weyl group $W_a$ and one can define similar  concepts as the ones defined in last section.  For example, $\theta(\lambda)$ corresponds to the alcove touching $\lambda$ in the direction of $\rho$ (the sum of the fundamental weights). The following theorem builds the bridge between Coxeter combinatorics and convex geometry.

\begin{lettertheorem}[Lattice Formula] \label{flf} For every dominant coweight $\lambda$, we have  $$\vert\leq\theta(\lambda)|=|W_f| \ |\mathrm{Conv}(W_f\cdot \lambda)\cap (\lambda+\mathbb{Z}\Phi^\vee)|.$$
\end{lettertheorem}
This formula is a key step to prove our main theorem below but it is also interesting in its own right, as we now explain. 
In \cite{postnikov2009permutohedra} Postnikov studied permutohedra of general types.
Among them, one of the most remarkable is the regular permutohedron of type $A_n$.
The number of integer points of that polytope can be interpreted \cite[Section 3]{stanley1980decompositions} as the number of forests on $\{1,2,\ldots, n\}$.
There are other interpretations for the integer points of the regular permutohedron of type $A_n$, for instance, \cite[Proposition 4.1.3]{backman2019geometric} gives one as certain orientations of the complete graph. We remark that these interpretations are only for the regular permutohedron of type $A_n.$
For non-regular permutohedra of any type, before the present paper, there was no interpretation of the integer points.  Theorem \ref{flf} gives a first interpretation of this sort, and it is also of a different nature than the pre-existent ones in that it is not related to graph theory but to Coxeter theory. 

 \newpage

This theorem also gives an interesting new insight. For a generic permutohedron (i.e. the convex hull of $\mathrm{Conv}(W_f\cdot \lambda)$ for some $\lambda \in \bbZ_{>0}\varpi_1+\bbZ_{>0}\varpi_2$), the set of vertices is in bijection with the finite Weyl group $W_f=\{w\leq_R w_0\}$ where $\leq_R$ is the right weak Bruhat order on $W_f$ and $w_0$ is the longest element.
The Hasse diagram of $\leq_R$ on $\{w\leq_R w_0\}$ corresponds to the graph of the polytope.

Theorem \ref{flf} (or, more precisely Proposition \ref{prop: to prove LF}) says that if we consider the strong Buhat order, the set $\leq \theta(\lambda)$  can be obtained from the lattice points inside the polytope.  Heuristically, the weak Bruhat order gives the vertices of the polytope and the strong Bruhat order gives the lattice points inside the polytope.

Now we can present the main result of this paper. 
For $J\subseteq \{1,2,\ldots, n\},$  one can define the face $F_J=\mathrm{Conv}(W_J\cdot \lambda)$. 
See section \ref{prelim} for more details. 

\begin{lettertheorem}[Geometric Formula]\label{geofor}
 For every rank 
$n$ root system $\Phi$, there are unique ${\mu_J^{\Phi}\in\mathbb{R}}$ such that for any dominant coweight $\lambda$,
$$|\leq\theta(\lambda)|=\sum_{J\subset \{1,\ldots,n\}}\mu_J^{\Phi}\mathrm{Vol}(F_J),
$$
\end{lettertheorem}

Theorem \ref{geofor} is proved by combining Theorem \ref{flf} with a particular formula for computing the number of lattice points developed by Berline-Vergne \cite{berline2007local} and Pommersheim-Thomas \cite{pommersheim2004cycles}.
The construction we use is part of a bigger family of formulae relating the number of lattice points of a polytope with the volumes of its faces, see \cite[Section 6]{barvinok1999algorithmic}.

In  \cite{postnikov2009permutohedra}, Postnikov gives several formulas for the volumes $\mathrm{Vol}(F_J)$ for any $\Phi$. When $\Phi$ is the root system of type $A_n$, we give in Section \ref{mu} formulas for $\mu_J^{\Phi}$ if $J$ is connected. For example,  
if $J=\{1,\ldots, l\}$, then 
\begin{equation} \label{eq: intro geo coeff}
\mu_J^{A_n}=\frac{l!}{\sqrt{l+1}}(n+1)\left[
\begin{matrix}
n+1 \\
l+1 \\
\end{matrix}
\right],   
\end{equation}
where the square bracket at the right means the Stirling number of the first kind.

The volumes are polynomials in the coordinates $m_1, \dots, m_n$ of $\lambda$ in the coweight basis.
As a consequence of Theorem \ref{geofor} we obtain that the size of the lower Bruhat intervals generated by $\theta(\lambda)$ is a polynomial function on the coordinates of $\lambda$.

\vspace{0.2cm}

\textbf{Perspectives:} As mentioned in the abstract, by some robust computational evidence we believe that one can partition the affine Weyl group \footnote{Technically one should be able to partition the affine Weyl group minus the set of parabolic subgroups isomorphic to the  finite Weyl group. The problem of finding cardinalities of intervals in the finite Weyl group seems to be a different kind of beast. Indeed, in the paper \cite[Thm 1.4]{dittmer2018counting} it is proven that the computation of lower intervals with respect to the weak Bruhat order is  $\#$P-complete, so for the weak Bruhat order no such partition of the finite Weyl group should be possible.}, with one of the parts being $\{\theta(\lambda)\}_{\lambda\in X^+},$ such that in each of these parts one will have a formula similar to that in Theorem \ref{geofor} but with different coefficients.

For the sets $\ngeq x:=\{w\in W\, \vert\, w\ngeq x\}$ we possess limited computational evidence, with the exception of the treacherous case of 
$\widetilde{A}_2.$ However, this does not prevent us from dreaming that a similar phenomenon (at least the polynomiality part) may occur for these sets, that are similar to star-shaped non-convex polytopes. If this was true, we would likely be on the verge of producing a formula for any interval, as $$\vert [x, y] \vert= \vert \leq y\vert- \vert\ngeq x\vert+P(x,y),$$ where $P(x,y)$ denotes the number of alcoves within a collection of simplices, easily computable in practice.

\subsection{Structure of the paper} Section 2 contains a recollection of definitions concerning affine Weyl groups and alcove geometry. Additionally, we establish the maximality of $\theta(\lambda)$ within a suitable double coset and provide the normalization for the volumes used in this paper. In Section 3 we present the proof of the lattice formula. Section 4 focuses on proving various results concerning the volumes of the $F_J$. These findings are then
employed to establish the Geometric formula in Section 5,
while in Section 6, we compute the $\mu_J^{A_n}$ for connected $J$. This last result relies on a formula by Luis Ferroni \cite{Ferroni}.

\subsection{Acknowledgments} 
We would like to thank Gaston Burrull, St\'ephane Gaussent, Mar\'ia In\'es Icaza, Jos\'e Samper, Joel Kamnitzer, Anthony Licata and Geordie Williamson for their helpful comments. Thanks to Daniel Juteau for helping with the redaction. Special thanks to Leonardo Patimo for some important insights. 
FC was partially supported by FONDECYT-ANID grant 1221133. 
NL was partially supported by FONDECYT-ANID grant 1230247. 
DP was partially supported by FONDECYT-ANID grant 1200341.

\section{Preliminaries}\label{prelim}

In this section we introduce the essential objects needed to state the Lattice Formula and the Geometric Formula. 
We refer to \cite{Bour46,humphreys1992reflection}  for more details about Weyl groups and to \cite{ziegler2012lectures} for more information about polytopes.


\subsection{Affine Weyl groups} 

Let $\Phi$ be an irreducible (reduced, crystallographic) root system of rank $n$, and let $V$ be the ambient (real) Euclidean space spanned by $\Phi$, with inner product ${(-,-)}: V \times V \to \mathbb{R} $.
Let $I_n\coloneqq\{1,\ldots,n\}$.

We fix a set ${\Delta=\{\alpha_i\mid i\in I_n\}}$ of simple roots and $\Phi^+$ is the corresponding set of positive roots.
Let ${\alpha^\vee=2\alpha/(\alpha,\alpha)}$ be the coroot corresponding to ${\alpha\in\Phi}$. 
The \textit{fundamental coweights} ${\varpi_i^\vee}$ are defined by the equations ${(\varpi_i^\vee,\alpha_j)=\delta_{ij}}$. 
They form a basis of $V$. 
A \textit{coweight} is an integral linear combination of the fundamental coweights, and a \textit{dominant coweight }is a coweight whose coordinates in this basis are non-negative. 
We denote by $\Lambda^\vee$ and ${(\Lambda^\vee)^+}$ the set of coweights and dominant coweights, respectively.  We define
\begin{equation*}
   C^+= \{ \lambda \in V \, | \, (\lambda,\alpha_i) \geq 0, \, \mbox{ for all }  i\in I_n  \}.
\end{equation*}
We refer to $C^+$ as the \textit{dominant region}.
We notice that ${(\Lambda^\vee)^+}= \Lambda^\vee \cap C^+  $. 
 
We denote by $\leq$ the \textit{dominance order} on $\Lambda^\vee$, that is, ${\mu\leq\lambda}$ if ${\lambda-\mu}$ can be written as a non-negative integral linear combination of simple coroots.

Let $H_{\alpha}$ be the hyperplane of $V$ orthogonal to a root $\alpha$, and let $s_\alpha$ denote the reflection through $H_{\alpha}$. 
For $\alpha_i\in \Delta$ we write $s_i=s_{\alpha_i}$. 
The group $W_f$ of orthogonal transformations of $V$ generated by $S_f=\{s_i\mid i\in I_n\}$ is the \textit{(finite) Weyl group} of $\Phi$.
It is a Coxeter system with generators $S_f$,  length function $\ell$ and Bruhat order $\leq$.  
We denote by $w_0$ the longest element of $W_f$. 

We also consider the \textit{affine Weyl group} $W_a$. 
It is the group of affine transformations of $V$ generated by $W_f$ and translations by elements of ${\mathbb{Z}\Phi^\vee}$, where ${\Phi^\vee}$ is the coroot system. 
We have ${W_a\cong \mathbb{Z}\Phi^\vee \rtimes W_f }$. The group $W_a$ can also be realized as the group generated by the affine reflections $s_{\alpha,k}$ along the hyperplanes 
\begin{equation*}
    H_{\alpha,k}=\{\lambda\in V\mid (\lambda,\alpha)=k\}, \text{ where } \alpha\in\Phi, k\in\mathbb{Z}.
\end{equation*}
Removing all these hyperplanes from $V$, leaves an open set whose connected components are called alcoves.
We choose the alcove
\begin{equation*}
    A_\text{id}\coloneqq\{\lambda\in V\mid -1<(\lambda,\alpha)<0, \ \forall\alpha\in\Phi^+\}
\end{equation*}
to be the \textit{fundamental alcove}. 
The map ${w\mapsto wA_\text{id}}$ defines a bijection between $W_a$ and the set of alcoves, so we define $ A_w := wA_\text{id}$ for each $w \in W_a$. 
We define the vertices of an alcove $A_w$ as the vertices of its closure $\overline{A_w}$.

The walls of $A_\text{id}$ are the hyperplanes $H_\alpha$ with ${\alpha\in\Delta},$ together with $H_{\tilde{\alpha},-1}$, where $\tilde{\alpha}$ is the highest root of $\Phi$.
We put ${s_0\coloneqq s_{\tilde{\alpha},-1}}$ and ${S\coloneqq S_f\cup\{s_0\}}$. 
Then the pair ${(W_a,S)}$ is a Coxeter system, with length function $\ell$ and 
Bruhat order $\leq$. 
For affine Weyl groups, we have a beautiful interpretation of the length function in terms of hyperplanes that separate a given alcove from the fundamental alcove.
More precisely, for $w\in W_a$ we have
\begin{equation} \label{separating hyp}
    \ell (w)  = \#\{ H=H_{\alpha,k } \mid  H \mbox{ separates } A_w \mbox{ from } \FA   \}.
\end{equation}

The \textit{extended affine Weyl }group,  $W_e$, is the subgroup of affine transformations of $V$ generated by $W_f$ and $\Lambda^\vee$ (acting as translations). 
We have ${W_e\cong \Lambda^\vee \rtimes W_f}$. 
In general, $W_e$ is not a Coxeter group. 
However, for every ${w\in W_e}$, ${w\FA}$ is still an alcove so that as in \eqref{separating hyp} one can define its length ${\ell(w)}$ by counting how many hyperplanes $H_{\alpha,k}$ separate $A_\text{id}$ and $wA_\text{id}$.

Let $\Omega$ be the subgroup of $W_e$ of length $0$ elements. 
Equivalently, $\Omega$ consists of the ${\sigma\in W_e}$  such that ${\sigma A_\text{id}=A_\text{id}}$. 
Thus the elements of $\Omega$ permute the walls of the fundamental alcove, so that conjugation by $\Omega$ permutes the simple reflections in $W_a$. 
In this way, $\Omega$ can be seen as a group of automorphisms of the corresponding completed Dynkin diagram. 
We define $W_\sigma\coloneqq\sigma W_f\sigma^{-1}$, which is isomorphic to $W_f$. 
We set ${s_\sigma\coloneqq \sigma s_0\sigma^{-1}}$, so that $W_\sigma$ is the maximal (finite) parabolic subgroup of $W_a$ generated by ${S\setminus\{s_\sigma\}}$. 

Other equivalent realization of this group is as a quotient: ${\Omega\cong\Lambda^\vee/\mathbb{Z}\Phi^\vee}$ (see \cite[\S 1.7]{IwahoriMatsumoto}). 
We will define a specific system of representatives of ${\Lambda^\vee/\mathbb{Z}\Phi^\vee}$.
Write the highest root as a combination of simple roots:
\begin{equation}
\label{eq:highest}
\tilde{\alpha}=\eta_1\alpha_1+\cdots+\eta_n\alpha_n.    
\end{equation}
One has that ${\eta_i\in\bbZ_{>0}}$.
For $i\in I_n$, it is not hard to check that the intersection of the reflecting hyperplanes corresponding to ${S\setminus\{s_i\}}$ 
 is ${v_i=-\varpi^\vee_i/\eta_i}$, for ${i\neq0}$, and $v_0=\mathbf{0}$ for $i=0$, where $\mathbf{0}$ is the origin of $V$. 
The set ${\{v_0,\ldots,v_n\}}$ is precisely the set of vertices of the fundamental alcove $A_\text{id}$. 
A fundamental coweight ${\varpi_i^\vee}$ is called \textit{minuscule} if ${(\varpi_i^\vee,\tilde{\alpha})=\eta_i=1}$. 
Let ${M\subset I_n}$ be the index set of the minuscule fundamental coweights. 
Both ${\{\mathbf{0},-\varpi_i^\vee\mid i\in M\}}$ and ${\{\mathbf{0},\varpi_i^\vee\mid i\in M\}}$ are complete systems of representatives of ${\Lambda^\vee/\mathbb{Z}\Phi^\vee}$.

It is known that for every ${\sigma\in\Omega\setminus\{\text{id}\}}$, the vector ${-\sigma(\mathbf{0})}$ is a minuscule fundamental coweight. 
Furthermore, ${\sigma\mapsto\sigma(\mathbf{0})}$ is a bijection from $\Omega$ to the representatives $\{\mathbf{0},-\varpi_i^\vee\mid i\in M\}$ of ${\Lambda^\vee/\mathbb{Z}\Phi^\vee}$ (see \cite[Prop~VI.2.3.6]{Bour46}).
Using the notation $v_i$ from the paragraph above, if $\sigma(\mathbf{0})=v_i$ with $\sigma\in\Omega$ then $s_\sigma=s_i\in S$, which is the unique simple reflection that does not fix $v_i$.
We will use this identification and put $\sigma$ instead of $\sigma(\mathbf{0})$, by abuse of notation.

The group $\Lambda^\vee$ contains $\mathbb{Z}\Phi^\vee$ as a subgroup of finite index, which is called the \textit{index of connection}. 
One can use it to compute the order of $W_f$ (see \cite[\S 4.9]{humphreys1992reflection}):
\begin{equation}\label{eq: wforder}
    |W_f|=n!\,\eta_1\cdots\eta_n\,[\Lambda^\vee:\mathbb{Z}\Phi^\vee].
\end{equation}




\subsection{Maximal elements in double cosets}
In this section we introduce the main protagonists of this paper.
Namely, the elements $\theta(\lambda)$ for $\lambda \in (\Lambda^\vee)^+$.
We also study some of their properties. 

\begin{definition} \label{def: theta}
  Let $\lambda$ be a dominant coweight. 
Since ${A_{w_0}+\lambda}$ is an alcove, there exists a unique element ${\theta(\lambda)\in W_a}$ such that ${A_{\theta(\lambda)}=A_{w_0}+\lambda}$. 
See Figure \ref{fig:tetha} for an example.  
\end{definition}


For any subset $J \subset S$, the subgroup $W_J$ of $W_a$ generated by $J$ is called a \textit{parabolic subgroup}.
We identify subsets of $S$ with subsets of $\{0,1,\ldots,n\}$.
In the following lemma we record some useful facts about parabolic double cosets in $W_a$ (see \cite[Lemma 2.12]{SingularCoxeter}).

\begin{lemma}\label{lem: double cosets}
Let $I$ and $J$ be proper subsets  of $S$ and  $p$ in $W_I\backslash W_a/W_J$. Then, 
\begin{enumerate}[(i)]
\item $p$ is an interval. That is, there exist $\underline{p}, \bar{p} \in p$ such that $p=\{ x\in W\mid \underline{p} \leq x \leq \bar{p}  \}$.  In particular, $p$ has a longest element.
\item The longest element $\bar{p}\in p$ is uniquely determined by the conditions
\begin{itemize}
    \item $I\subset  \{  s\in S \mid \ell(s\bar p) <\ell(\bar p)\}$.
    \item $J\subset  \{  s\in S \mid \ell(\bar p s ) <\ell(\bar p)  \}$.
\end{itemize}
\end{enumerate}    
\end{lemma}

\begin{definition}
 For any  ${X\subset W_a}$ we define
\begin{equation*}
    A(X)\coloneqq \bigsqcup_{x\in X}A_x.
\end{equation*}   
\end{definition}

\begin{lemma}\label{lemmafirstlat}
Let $\lambda$ be a dominant coweight and let ${\sigma\in\Omega}$ such that ${\lambda\in\sigma+\mathbb{Z}\Phi^\vee}$. 
Then,
\begin{enumerate}[(i)]
    \item \label{uno}${A(W_\sigma)=A(W_f)+\sigma}$.
    \item \label{dos}  ${A(\theta(\lambda)W_\sigma)=A(W_f)+\lambda}$.
    \item \label{tresA} $\theta(\lambda)$ is maximal with respect to the Bruhat order in its right coset $W_f\theta(\lambda) $. 
    \item \label{tres}  $\theta(\lambda)$ is maximal with respect to the Bruhat order in its left coset $\theta(\lambda) W_{\sigma}$. 
    \item \label{cinq}  $\theta(\lambda)$ is maximal with respect to the Bruhat order in its double coset $W_f\theta(\lambda) W_{\sigma}$.
\end{enumerate}
\end{lemma}

\begin{proof}\hfill
\begin{enumerate}[(i)]
    \item  It is known that the alcoves corresponding to $W_f$ are precisely the alcoves having the origin $\mathbf{0}$  as one of its vertices. Since $\sigma\in\Lambda^\vee$ (under the identification $\sigma\mapsto\sigma(\mathbf{0})$), we get                 \begin{equation*}
        A(W_f)+\sigma=\{\text{alcoves that have }\sigma\text{ as one of its vertices}\}.
    \end{equation*}
    Now let $w\in W_f$. Note that $\sigma w\sigma^{-1}A_\text{id}=\sigma wA_\text{id}$, so that ${A(W_\sigma)=\sigma(A(W_f))}$, by definition of $W_\sigma$. It follows that 
    \begin{equation*}
        A(W_\sigma)\subset\{\text{alcoves that have }\sigma\text{ as one of its vertices}\}.
    \end{equation*}
    That is, $A(W_f)+\sigma$ is a collection of $|W_f|$ alcoves containing $A(W_\sigma)$. Since $W_\sigma\cong W_f$, the set $A(W_\sigma)$ also has exactly $|W_f|$ alcoves. Thus ${A(W_\sigma)=A(W_f)+\sigma}$.
    
    \item Write ${\lambda=\sigma+\mu}$ for $\mu \in \mathbb{Z}\Phi^\vee$. Let  ${t_\mu\in W_a}$ be the translation by $\mu$. We notice that
\begin{equation*}
   A_{t_\mu^{-1}\theta(\lambda) } =  t_\mu^{-1}\theta(\lambda)(\FA) = t_\mu^{-1}\left(A_{w_0}+\lambda\right)=\left(A_{w_0}+\lambda\right)-\mu=A_{w_0}+\sigma.
\end{equation*}
It follows that $A_{t_\mu^{-1}\theta(\lambda)} \in A(W_f)+\sigma$.    By \eqref{uno} we conclude that ${t_\mu^{-1}\theta(\lambda)\in W_\sigma}$. Thus,
\begin{equation*}
    A(\theta(\lambda)W_\sigma) = A(t_\mu W_\sigma) = t_\mu A(W_\sigma) = A(W_\sigma)+\mu.
\end{equation*}

\item We will use \eqref{separating hyp} in order to show that $\ell(s\theta(\lambda)) < \ell (\theta(\lambda))$ for all $s\in S_f$. Notice that the claim then follows by applying Lemma \ref{lem: double cosets} for $I=S_f$ and $J=\emptyset$. 

We will prove that if there is a hyperplane $H_{\alpha, k}$ separating $A_{\text{id}} $  from  $A_{s\theta(\lambda)}$, then  $sH_{\alpha, k}$ separates $A_{\text{id}} $  from  $A_{\theta(\lambda)}$. Let $H=H_{\alpha, k}$ be a hyperplane that separates $A_{\text{id}} $  from  $A_{s\theta(\lambda)}$.  Suppose that $sH$ does not separate  $A_{\text{id}} $  from  $A_{\theta(\lambda)}$.  Then,  $sH$ separates $A_{\text{id}}$ from $A_{s}$. Since $H_{\alpha_s}$ is the unique hyperplane that separates $A_{\text{id}}$ from $A_{s}$, we conclude $sH=H=H_{\alpha_s} $.  However, since $A_{\theta(\lambda)}\subset C^+$ we know that $H_{\alpha_s}$ separates $A_{\text{id}} $  from  $A_{\theta(\lambda)}$. This contradiction proves our claim.

\item For $x\in W_f$ we denote by $x'$ the unique element 
$x'\in W_a$ such that $A_{x'} = A_x+\lambda$. 
We claim that 
$\ell(x')\leq \ell (w_0')$ for all $x\in W_{f}$. We will prove this by 
showing that each hyperplane $H_{\alpha,k}$ that separates $A_{x'}$ 
from $\FA$ must also separate $A_{w_0'}$ from $\FA$.
We proceed by 
contradiction. 
Suppose there is an hyperplane  $H'$ as above that separates $A_{x'}$
from $\FA$ but it does not separate $A_{w_0'}$ from $\FA$. 
Thus $H'$ separates $A_{w_0'}$ from $A_{x'}$, but these alcoves share the vertex $\lambda,$ so that $\lambda\in H'$.
Then, the hyperplane  $H=H'-\lambda$
passes through the origin and separates $A_{x}$  from $A_{w_0}$. As any $H_{\alpha,0}$ separates $A_{\text{id}}$ from $A_{w_0}$,
 the alcoves $A_{x}$ and $\FA$ are on the same side of $H$. 
Therefore, $A_{x'}$ and $A_{\text{id}'}$ are on the same side of $H'$. 
Since $\lambda$ is dominant, $A_{\text{id}}$ and $A_{\text{id}'}$  are 
on the same side of $H'$. Thus $A_{x'}$ and $A_{\text{id}}$ are on the 
same side of $H'$, which contradicts our choice of $H'$ and proves our 
claim. Since $\theta(\lambda) = w_0'$ we conclude from \eqref{dos} that $\ell (x') \leq \ell(\theta(\lambda))$ for all $x'\in \theta(\lambda) W_{\sigma}$. The result now follows by applying Lemma \ref{lem: double cosets} for $I=\emptyset$ and $J=S\setminus \{s_{\sigma}\}$.

\item This follows by combining Lemma \ref{lem: double cosets} for $I=S_f$ and $J=S\setminus \{s_{\sigma}\}$ together with \eqref{tresA} and $\eqref{tres}$. \qedhere
\end{enumerate}
\end{proof}

\subsection{Polytopes and their volumes} \label{section poly and volumes }

It is a classic result that there exists a unique translation invariant measure $ \mu $ on $ \mathbb{R}^m $ up to scaling.
The Euclidean volume on $\mathbb{R}^m$ is a translation invariant measure $\vol_m$ normalized so that $ \vol_m([0,1]^m) = 1 $,
 where $[0,1]^m$ is the Cartesian product of $m$ unit segments, i.e. the unit cube.
In the present paper we assume that every vector space $V$ lives inside some $\mathbb{R}^m$ that comes equipped with an Euclidean volume.
The volume in $\mathbb{R}^m$ has the property that $|\det (v_1,\dots,v_m)| = \vol_m(\Pi_{v_1,\dots,v_m}),$ where $$\Pi_{v_1,\dots,v_m} = \{a_1v_1 + \dots + a_mv_m \in \mathbb{R}^m \mid  0\leq a_i \leq 1, \,  \forall\  1\leq i \leq m  \}$$ is the parallelepiped spanned by the vectors $\{v_1,\cdots,v_m\}$.

\begin{definition}
\label{def:lattice}
A lattice is a discrete subgroup $ \Gamma $ of $ V $ that spans $V$ as a vector space.
As a group, a lattice is always isomorphic to $ \mathbb{Z}^d $ where $d=\dim(V)$ \cite[Theorem 10.4]{barvinok2008integer}.
We define the \emph{determinant} of $\Gamma$ as the volume of $\Pi_{v_1,\dots,v_m}$ for any integral basis $v_1,\dots, v_m$ of $\Gamma$.
The determinant of $\Gamma$ does not depend on the integral basis \cite[Theorem 10.8]{barvinok2008integer}.
\end{definition}

\begin{example}
\label{ex:alcove_volume}
The alcove $A_{\text{id}}$ is a simplex with vertices $\mathbf{0},-\varpi^\vee_1/\eta_1,\dots,-\varpi^\vee_n/\eta_n$, where the numbers $\eta_i$ are defined in Equation \eqref{eq:highest}. 
Thus, we have
\begin{equation}
\label{eq:alcove_volume}
\vol (A_{\text{id}}) = \frac{\det(\Lambda^\vee)}{n!\eta_1\cdots \eta_n}.
\end{equation}
\end{example}

A priori the volume of a subset contained in a proper subspace of $\mathbb{R}^m$ is zero.
However we can consider an induced volume on any subspace $V$ as follows.
Let $\{v_1,\dots,v_k\}$ be a basis for the subspace $V$.
The Euclidean volume induces a measure $\vol_k$ (also called Euclidean volume, by abuse of notation) on $V$ defined by the property that
$\vol_k (\Pi_{v_1,\dots,v_k} )= \det(v_1,\dots,v_k,u_{k+1},\dots,u_m)$, where $\{u_{k+1},\dots,u_m  \}$ is an orthonormal basis for the orthogonal complement of $V$ in $\mathbb{R}^m$.

An important part of this paper focuses on the study of volumes of polytopes living in  the ambient space, $V$, of a given root system $\Phi$. In this setting, we embed $V$ inside some $\mathbb{R}^m$ by following the 
conventions outlined in \cite[Plates I,\ldots,VI]{Bour46}. In particular, we have an explicit description for $V$ within a specific $\mathbb{R}^m$, accompanied by explicit descriptions for roots, coroots, coweights, etc. In the following subsection, we give all the details in type A.

\subsubsection{Type A}
\label{ex:typeA}
Let $\Phi$ be a root system of type $A_n$. 
In this case $V$ is the hyperplane of $\mathbb{R}^{n+1}$ (with standard basis $\varepsilon_1,\ldots,\varepsilon_{n+1}$) of vectors whose coordinate sum is zero and
$  \Phi = \{  \varepsilon_i-\varepsilon_j  \mid 1\leq i , j \leq n+1, \, i\neq j   \}$.  
The simple roots are given by $\alpha_i = \varepsilon_i-\varepsilon_{i+1}$ with $1\leq i\leq n$, and the positive roots are the vectors $\varepsilon_i-\varepsilon_j$ with $1\leq i < j\leq n+1$.

In this example the parallelepiped $\Pi_\Delta \subseteq V$ spanned by the simple roots has $n$-Euclidean volume equal to
\small
\begin{equation}
\label{eq:vol_fundamental_root}
\det(\mathbb{Z}\Phi)=\vol_n(\Pi_\Delta) = 
\det\left|\begin{array}{cccccc}
1&-1&0&\cdots &0&0\\
0&1&-1&\cdots&0&0\\
\vdots&\vdots&\vdots&\ddots&\vdots&\vdots\\
0&0&\cdots&\cdots&1&-1\\
\frac{1}{\sqrt{n+1}}&\frac{1}{\sqrt{n+1}}&\frac{1}{\sqrt{n+1}}&\ldots &\frac{1}{\sqrt{n+1}}&\frac{1}{\sqrt{n+1}}  \\
\end{array}\right|=\frac{n+1}{\sqrt{n+1}} = \sqrt{n+1},
\end{equation}
\normalsize
as can be checked by using row operations to transform the last row into 
$$\left[0,\cdots,0,\frac{n+1}{\sqrt{n+1}}\, \right].$$

Since $(\alpha, \alpha) =2$ for all $\alpha \in \Phi$ , we have   $\alpha=\alpha^\vee$.
Therefore the fundamental weights and fundamental coweights coincide, and are given by
\begin{equation}
\label{eq:fundamenta_coweight}
    \varpi_{i} =  \varepsilon_1+\cdots+\varepsilon_i - \dfrac{i}{n+1}(\varepsilon_1+\cdots+\varepsilon_{n+1}).
\end{equation}
The volume of the fundamental alcove is $\sqrt{n+1}/(n+1)!$.
This follows from Example \ref{ex:alcove_volume} and a computation of $\det(\Lambda^\vee)$ using the coweights in Equation \eqref{eq:fundamenta_coweight}.

\subsubsection{Orbit polytopes}

A \emph{polytope} $ \mathsf{P} \subseteq V $	is the convex hull of finitely many points.
A \emph{supporting hyperplane} $H$ of a polytope $\mathsf{P}$ is an affine hyperplane such that $ \mathsf{P} \cap H \neq \emptyset $ and $ \mathsf{P} $ is contained in one of the two closed halfspaces defined by $ H $.
A \emph{face} $ \mathsf{F} \subseteq \mathsf{P} $ is the intersection of $ \mathsf{P} $ with a supporting hyperplane.
We also consider the whole polytope and the empty set to be faces.
Faces of dimension 0 and 1 are called \emph{vertices} and \emph{edges} respectively.
Faces of codimension 1 are called \emph{facets}.

\begin{definition}
\label{def:orbit_polytope}
Let $\Phi$ be an irreducible root system with simple roots $\Delta=\{\alpha_1,\dots, \alpha_n\}$ and let $ \lambda \in V$.
The orbit polytope $ \mathsf{P}^\Phi(\lambda) $ of an element  $ \lambda \in V$ is defined as the convex hull of the  $W_f$-orbit of $\lambda,$ i.e. $\conv\{ w\cdot \lambda \mid w\in W_f \} $.
Without loss of generality we always assume that $\lambda$ is in the dominant region $C^+$. 
If $\lambda =\mathbf{0}$ then $ \mathsf{P}^\Phi(\lambda) = \{  \mathbf{0}\} $. 
Otherwise, $ \mathsf{P}^\Phi(\lambda) $ is full dimensional. 
\end{definition}

Let $\lambda = m_1\varpi_1^\vee+\dots+ m_n\varpi_n^\vee \in C^+$.
We define the \emph{vanishing set} of $\lambda$ as 
$$Z(\lambda) :=\{j\in I_n \mid m_j = 0\}. $$
For $j\in I_n$ the element $s_j \in W_f$ fixes $\lambda$ if and only if $j \in Z(\lambda)$.
Furthermore, the stabilizer of $\lambda$ is the parabolic subgroup $W_{Z(\lambda)}$.
The face structure of the orbit polytopes depends on vanishing sets as the following theorem (proved in \cite[Corollary 1.3]{renner2009descent}) makes precise.

\begin{prop}
\label{prop:face_lattice}
Let $\Phi$ be a root system with Dynkin diagram $D.$ Let
$ \lambda $ be an element with vanishing set $ Z $.
There is a bijection between
\begin{enumerate}[(i)]
\item $ W_f $-orbits of $d$-dimensional faces of $ \mathsf{P}^\Phi(\lambda)$
\item Subsets $ J \subseteq \Delta $ with $|J|=d$ such that no connected component of $ D|_J $ is contained in $ D|_Z $. \label{itm:condition_K}
\end{enumerate}
\end{prop}

We can say  more about this bijection.

\begin{definition}
\label{def:face_representative}
For $J$ a subset of $\Delta$ satisfying Theorem \ref{prop:face_lattice}~(ii), let $\mathsf{F}_J(\lambda)$ be the unique face in the corresponding $W_f$-orbit  containing $\lambda$.
\end{definition}

We can describe $\mathsf{F}_J(\lambda)$ in two ways.
First we have that
\begin{equation}
\label{eq:face_J_vertices}
\mathsf{F}_J(\lambda) := \conv\{w\cdot\lambda \mid  w\in W_J\}.
\end{equation}

We can also describe $F_J$ as an intersection of $\mathsf{P}^\Phi(\lambda)$ with supporting hyperplanes.
For each index $j \in I_n$ we have that the hyperplane
\begin{equation}\label{eq:Hj}
H_j(\lambda)\coloneqq \{ v \in V \mid \langle \varpi_{j}^\vee, v\rangle = \langle \varpi_{j}^\vee, \lambda\rangle \}
\end{equation}
is a supporting hyperplane of $\mathsf{P}^\Phi(\lambda)$. 
For a set $ J \subset I_n $ we define the affine subspace 
\begin{equation}
\label{eq:face_span}
H_{J}(\lambda) := \bigcap_{j\in J^c} H_j(\lambda) = \lambda + \text{Span}\{\alpha_j\, :\, j\in J\},
\end{equation} where $J^c$ is the complement of $J$ in $I_n$.
Note that $H_j(\lambda)=H_{I_n\setminus\{j\}}(\lambda)$.
By the second equality we can see that the linear subspace parallel to $H_J(\lambda)$ is
\begin{equation}
    \label{eq:face_linear}
    L_J = \text{Span}\{\alpha_j\,:\,j\in J\}.
\end{equation}
We have that
\begin{equation}
\label{eq:face_J_ineqs}
\mathsf{F}_{J}(\lambda) \coloneqq H_{J}(\lambda) \cap  \mathsf{P}^\Phi(\lambda).
\end{equation}

\begin{remark}
The definition of orbit polytopes makes sense even for non-irreducible root systems.
In particular, Equation \eqref{eq:face_J_vertices} shows that the face $\mathsf{F}_J(\lambda)$ is itself an orbit polytope
with Weyl group $W_J$. 
Proposition \ref{prop:face_lattice} and Corollary \ref{cor:generic} are valid not just for orbit polytopes but also for their faces.

\end{remark}

\begin{remark}
\label{rem:facets}
Every facet containing $\lambda$ is an intersection of the form $ \mathsf{P}^\Phi(\lambda) \cap H_i $.
However, when $i\in Z(\lambda)$  the intersection is a face of codimension greater than 1.
We call such faces degenerate facets.
\end{remark}

An immediate consequence of Proposition \ref{prop:face_lattice} is the following.
\begin{cor}
\label{cor:generic}
Let $ \lambda $ be \emph{generic}, i.e., with empty vanishing set. Then the
$ W_f $-orbits of faces of $ \mathsf{P}^\Phi(\lambda)$ are in bijection with subsets of $ \Delta $.
In particular, the $ W_f $-orbits of the facets are in bijection with $ \Delta $. 
More precisely, every
face of $ \mathsf{P}^\Phi(\lambda)$ is in the $W_f$-orbit of $F_J(\lambda)$ for some $J\subset I_n$.
Furthermore, the face $F_J(\lambda)$ has dimension $|J|$ and its $W_f$-orbit consists precisely  of $[W_f:W_J]$ faces. 
\end{cor}

\section{Lattice Formula }
For any ${x,y\in W_a}$, let ${[x,y]}$ be the Bruhat interval consisting of the elements ${z\in W_a}$ such that ${x\leq z\leq y}$. 
We write ${\leq y\coloneqq[\text{id},y]}$. 

\begin{prop} \label{prop: to prove LF}
  For every dominant coweight $\lambda$, we have 
  \begin{equation}\label{leqthetaalcoves}
    A\bigg(\leq\theta(\lambda)\bigg)=\bigsqcup_{\mu\in W_f\cdot X_\lambda}A(W_f)+\mu,
\end{equation}
where ${X_\lambda=\{\mu\in(\Lambda^\vee)^+\mid\mu\leq\lambda\}}$.
\end{prop}

\begin{proof}
Let ${\sigma\in\Omega}$ be such that ${\lambda\in\sigma+\mathbb{Z}\Phi^\vee}$. By the Lifting Property and \eqref{cinq} in Lemma \ref{lemmafirstlat}, we can easily prove that the set ${\leq\theta(\lambda)}$ is $W_f$-invariant on the left and $W_\sigma$-invariant on the right. On the other hand, for every ${\sigma'\in\Omega}$ the map $\theta$ 
\begin{equation}\label{biyec}
\theta:\left(\sigma'+\mathbb{Z}\Phi^\vee\right)\cap(\Lambda^\vee)^+\xrightarrow{\sim}W_f\backslash W_a/W_{\sigma'},
\end{equation}
given by $\lambda' \mapsto W_f \theta(\lambda')W_{\sigma'}$ is a bijection that intertwines the dominance order on the left with the Bruhat order on the right (for more details on this bijection see \cite[Section 2.1]{PreCanonical}).  In particular, if $\mu, \lambda \in (\Lambda^\vee)^{+}, $ then ${\mu\leq\lambda}$ if and only if $\mu\in\sigma+\mathbb{Z}\Phi^\vee$ and ${\theta(\mu)\leq\theta(\lambda)}$. 

Let us prove the equation 
\begin{equation}\label{desc}
\leq\theta(\lambda)=\bigsqcup_{\mu\in X_\lambda}W_f\theta(\mu)W_\sigma.
\end{equation}

First, we prove the inclusion $\supseteq$. If $\mu \in X_{\lambda},$ then ${\theta(\mu)\leq\theta(\lambda)}$ and thus
 the invariance of ${\leq\theta(\lambda)}$ implies that ${W_f\,\theta(\mu) W_\sigma}$ is contained in the set ${\leq\theta(\lambda)}$. We now prove the inclusion $\subseteq.$
Let ${u\leq\theta(\lambda)}$. Once again,  the invariance of $\leq \theta (\lambda)$ implies that  ${W_f\,uW_\sigma} $ is contained in the set $\leq \theta(\lambda) $. By the bijection \ref{biyec}, the maximal element of the coset ${W_f\,uW_\sigma}$ is ${\theta(\mu)}$ for some $\mu\in\left(\sigma+\mathbb{Z}\Phi^\vee\right)\cap(\Lambda^\vee)^+$.   Then, ${W_f\,\theta(\mu)W_\sigma} $ is contained in the set $\leq \theta(\lambda) $. In particular, ${\theta(\mu)\leq\theta(\lambda)}$ and thus ${\mu\in X_\lambda}$. We conclude the proof of equality  \eqref{desc}. From there we see
\begin{equation*}
    \leq\theta(\lambda) =W_f\cdot\left(\bigsqcup_{\mu\in X_\lambda}\theta(\mu)W_\sigma\right).
\end{equation*}
By looking at the corresponding alcoves and by using  Lemma \ref{lemmafirstlat}\eqref{dos}, we get
\begin{align*}
   A(\leq\theta(\lambda))&=W_f\cdot\left(\bigsqcup_{\mu\in X_\lambda}A(\theta(\mu)W_\sigma)\right)\\
    &=W_f\cdot\left(\bigsqcup_{\mu\in X_\lambda}A(W_f)+\mu\right)\\
    &=\bigsqcup_{\mu\in W_f\cdot X_\lambda}A(W_f)+\mu.
\end{align*}
\end{proof}

The closure of the set ${A(W_f)\subset V}$ defines a polytope, which we call the $W_f$-polytope. 
If $\Phi$ has type $A_n$, then $W_f$ is isomorphic to the symmetric group $S_{n+1}$. 
Figure \ref{polytope example} shows the $S_3$-polytope and the $S_4$-polytope. 
The black arrows are the fundamental (co)weights and the colored arrows are the simple (co)roots. 
Proposition \eqref{leqthetaalcoves} shows that the $W_f$-polytope tessellates the closure of the set $A(\leq \theta(\lambda))$.

\begin{figure}[ht]
\captionsetup[subfigure]{labelformat=empty}
  \begin{subfigure}[b]{0.5\textwidth}
    \includegraphics[width=\textwidth]{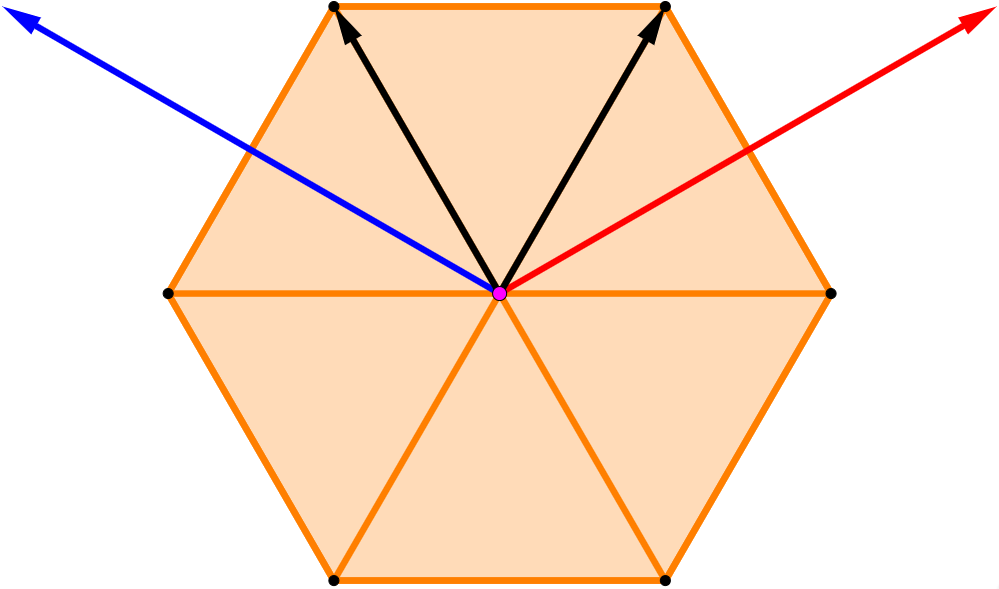}
    \caption{$S_3$-polytope}
  \end{subfigure}
  \hfill
  \begin{subfigure}[b]{0.4\textwidth}
    \includegraphics[width=\textwidth]{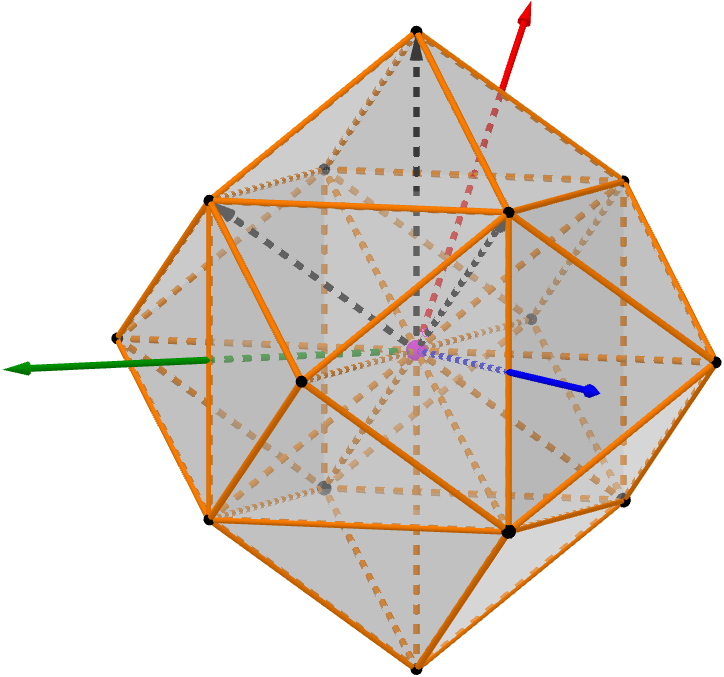}
    \caption{\hspace*{2em}$S_4$-polytope}
  \end{subfigure}
  \caption{$W_f$-polytope for types $\tilde{A_2}$ and $\tilde{A_3}$}
  \label{polytope example}
\end{figure}

\begin{theorem}[Lattice Formula]\label{thm:first-lattice-formula}
For every dominant coweight $\lambda$, we have 
\begin{equation}\label{firstlatdef}
    |\leq\theta(\lambda)|=|W_f| \ |\mathsf{P}^\Phi(\lambda)\cap (\lambda+\mathbb{Z}\Phi^\vee)|.
\end{equation}
\end{theorem}

\begin{proof}
    Note that $\mathsf{P}^\Phi(\lambda)\cap C^+$ consists precisely of the elements $\lambda-(x_1\alpha_1^\vee+\cdots +x_n\alpha_n^\vee)\in C^+$ with all the $x_i\geq 0$. 
It follows that $X_\lambda=\mathsf{P}^\Phi(\lambda)\cap(\lambda+\mathbb{Z}\Phi^\vee)\cap C^+$. 
Since ${\mathsf{P}^\Phi(\lambda)}$ and ${\lambda +\mathbb{Z}\Phi^\vee}$ are $W_f$-invariant,  and ${V=W_f\cdot C^+}$, we obtain
\begin{equation*}
    W_f\cdot X_\lambda=\mathsf{P}^\Phi(\lambda)\cap(\lambda+\mathbb{Z}\Phi^\vee).
\end{equation*}
Therefore, \eqref{firstlatdef} follows by counting alcoves in \eqref{leqthetaalcoves}.
\end{proof}

\section{On the volumes \texorpdfstring{$V_J^\Phi$}{}}


In this section we study the volumes of orbit polytopes.
We fix an irreducible root system $\Phi$ of rank $n$.
Notice that faces of $\mathsf{P}^\Phi(\lambda)$ in the same $W_f$-orbit have the same volume, without loss of generality we can focus on the representatives $\mathsf{F}_J(\lambda)$. 

\begin{definition}
    For $\lambda\in V$ and $J\subset I_n$, we define ${V_J^{\Phi}(\lambda)}$ as the ${|J|}$-dimensional volume of $\mathrm{Conv}(W_J\cdot\lambda)$. 
\end{definition}


Let $D$ be the Dynkin diagram corresponding to $\Phi$. We denote by  $D_J$ the graph obtained from  $D$ by eliminating all the vertices $i$ for $i\not \in J$. We define $\mathcal{C}_J$ as the collection of the index sets of the connected components of $D_J$. For example in type $A_n$ for $n\geq 4$, $\mathcal{C}_{\{1,2,4\}}=\{\{1,2\},\{4\}\}$. We say that $J$ is connected if $\mathcal{C}_J=\{J\}$.

In the following lemma we record some basic facts about $V_J^\Phi(\lambda)$. 

\begin{lemma}\label{lem: volgralfacts}
    Let ${J\subset I_n}$ and let ${\lambda=m_1\varpi^\vee_1+\cdots+m_n\varpi^\vee_n\in V}$.
    \begin{enumerate}[(i)]
        \item  \label{VolFact1} The function $V_J^\Phi$ only depends on the ``$J$-coordinates''. That is, if ${\lambda_J=\displaystyle\sum_{j\in J}m_j\varpi_j^\vee}$ then 
          $$  V_J^\Phi(\lambda)=V_J^\Phi(\lambda_J).$$
        \item  \label{VolFact2} The volume ${V_J^\Phi(\lambda)}$ can be computed as the product of the volumes corresponding to the connected components of $J$. That is,
           $$ V_J^\Phi(\lambda)=\prod_{K\in \mathcal{C}_J}V_K^\Phi(\lambda).$$
   \qed
    \end{enumerate}
\end{lemma}



We will give a recursive formula for $V_J^\Phi$. Before doing so, we will need some previous results. 

\begin{definition}
    For $J\subset I_n$ we define $\mathcal{B}_J=\{b_1,\ldots , b_n\}$ where
\begin{equation*}
    b_i=\begin{cases}
        \alpha_i^\vee & \text{if } i\in J \\
        \varpi_i^\vee & \text{if } i\notin J
    \end{cases}
\end{equation*}   
We will show that this set is in fact a basis. We call $\mathcal{B}_J$ the $J$-mixed basis. For $j\in I_n$ we define $\nu_j\in V$ to be the unique element satisfying $(\nu_j,b_i)=\delta_{ij}$ for all $i\in I_n$. 
\end{definition}

\begin{lemma}\label{lem: pyramidprev} Let $J\subset I_n$. Then,

\begin{enumerate}[(i)]
    \item  \label{basis JJ}
 $\mathcal{B}_J=\{b_1,\dots,b_n\}$ is a basis of $V$.
\item \label{trestres}  For all $j\in J$ we have $(\varpi_j^\vee,\nu_j)>0$. 
\item  \label{zerozero} For any $x\in L_J$ we have  $ \displaystyle  \sum_{w\in W_J}wx=0$, where $L_J$ is given by \eqref{eq:face_linear}. 
\end{enumerate}
\end{lemma}

\begin{proof}\hfill
\begin{enumerate}[(i)]
    \item We proceed by induction on $|J^c|$. If $J^c=\emptyset $ there is nothing to prove.  Let us now assume that $|J^c|\geq 1$. Let $i\in J^{c}$ and define $J_i^c=J^c\setminus\{i \}$.  Of course, $|J_i^c|<|J^{c}|$. Thus our inductive hypothesis implies that $\mathcal{B}_{J_i}$ is a basis of $V$. Thus we can write  
    \begin{equation} \label{varpi in mixed basis}
       \varpi_i^\vee   = \sum_{j\in J_i} \lambda_{j}\alpha_{j}^\vee + \sum_{k\in J_i^c} \lambda_{k}\varpi_{k}^\vee.
    \end{equation}
    If $\lambda_i\neq 0$ then we are done, since in this case $\alpha_{i}^\vee $ would live in the span of $\mathcal{B}_J$. So, we assume that $\lambda_i=0$. After pairing both sides of \eqref{varpi in mixed basis} with $\alpha_j$ for $j\in J$ we get a homogeneous linear system with $|J|$ unknowns $\{\lambda_j\}_{j\in J}$ whose matrix is obtained from the Cartan matrix of $\Phi$ by eliminating all the rows and columns indexed by $J^c$. We denote this matrix by $M_J$. We notice that $M_{J}$ is a block diagonal matrix with each block being the Cartan matrix of the root system associated to the connected components of $J$. 
     It follows that  $M_{J}$ is invertible. We conclude that $\lambda_j=0$ for all $j\in J$. Therefore, \eqref{varpi in mixed basis} reduces to 
     \begin{equation} \label{varpi in mixed basis two}
       \varpi_i^\vee   =  \sum_{k\in J_i^c} \lambda_{k}\varpi_{k}^\vee.
    \end{equation}
    This is a contradiction since $\{ \varpi_k^\vee \}_{k\in I_n}$ is a basis of $V$. It follows that $\lambda_i\neq 0$ and $\mathcal{B}_J$ is a basis of $V$.


\item Fix $j\in J$. Note that $\nu_j$ can equivalently be defined as the vector $\nu_j\in L_J$ such that $(\nu_j,\alpha_i^\vee)=\delta_{ij}$ for all $i\in J$. Write 
\begin{equation}\label{eq: claimnu}
    \nu_j=\sum_{k\in J}u_k\alpha_k,
\end{equation}
for some real numbers $u_k$. Pairing each side of \eqref{eq: claimnu} with $\alpha_i^\vee$ for $i\in J$, gives a system of equations $e_j=M_J\,u$, where $e_j=(\delta_{ij})_{i\in J}$ and $u=(u_i)_{i\in J}$ are column vectors.   It follows that $u$ is the $j^\text{th}$ column of $M_J^{-1}$. Furthermore, by \cite[\S5]{lusztig1992inverse} we know that all the entries of $M_J^{-1}$ are strictly positive. We conclude that $u_k>0$ for all $k\in J$. In particular, $0<u_j=(\varpi_j^\vee,\nu_j)$. 
    \item By linearity we can assume $x=\alpha_j$ for some $j\in J$. The group $W_J$ can be partitioned as
    \begin{equation}
        W_J = W_J^{s_j} \sqcup W_J^{s_j}\cdot s_j,
    \end{equation}
    where $W_J^{s_j}  = \{  w\in W_J \mid \ell(w) <\ell(ws_j)  \} $.  By recalling that $s_j(\alpha_j) = -\alpha_j$, we obtain 
\begin{equation}
  \sum_{w\in W_J}wx  = \sum_{w\in W_J^{s_j} }  (wx+ws_jx) = \sum_{w\in W_J^{s_j} }  (wx-wx) = 0. \qedhere
\end{equation}
\end{enumerate}
\end{proof}

\begin{lemma}\label{lem: generalized pyramid formula} Let $J\subset I_n$.  For any $\lambda\in C^+$ we have
\begin{equation}\label{eq: generalized pyramid formula}
    V_J^\Phi(\lambda)=\dfrac{1}{|J|}\sum_{j\in J}\left[W_J:W_{J\setminus\{j\}}\right]\,\dfrac{(\lambda, \nu_j)}{\Vert\nu_j\Vert} \ V_{J\setminus\{j\}}^\Phi(\lambda).
\end{equation}    
\end{lemma}

\begin{proof}
Let $\lambda=(m_1,\ldots,m_n)$ in the fundamental coweight basis. Note that $W_J\cdot\lambda\subset L_J+\lambda$. We will compute the $|J|$-dimensional volume of $\mathcal{P}=\conv(W_J\cdot\lambda)$ in the space $U=L_J+\lambda$. For $w\in W_J$ and $j\in J$, let
\begin{equation*}
F(w,j)=w\conv(W_{J\setminus\{j\}}\cdot\lambda).
\end{equation*}
Let us suppose that $m_i>0$ for all $i\in J$, so that, by Corollary \ref{cor:generic}, the facets of $\mathcal{P}$ are given by
\begin{equation*}
    \mathcal{F}=\{F(w,j)\mid j\in J \text{ and } w\in W_J/W_{J\setminus\{j\}}\}.
\end{equation*}
For $K\subset I_n$, let $\xi_K\in V$ be the vector resulting from writing $\lambda$ in the $J$-mixed basis and then removing all but its $K$-coordinates.
By definition, it follows that $\xi_{J^c}$ is stabilized by $W_J$, where $J^c$ is the complement of $J$ in $I_n$.
Since $\xi_J\in L_J$, Lemma \ref{lem: pyramidprev}\eqref{zerozero} gives
\begin{equation*}
    \dfrac{1}{|W_J|}\sum_{w\in W_J}w\lambda=\dfrac{1}{|W_J|}\sum_{w\in W_J}w(\xi_{J^c}+\xi_J)
    =\xi_{J^c}+\dfrac{1}{|W_J|}\sum_{w\in W_J}w\xi_J
    =\xi_{J^c}.
\end{equation*}
Therefore, $\xi_{J^c}\in\mathcal{P}$. Let $h(w,j)$ be the pyramid in $U$ having $F(w,j)$ as its base and $\xi_{J^c}$ as its apex. Note that $\mathcal{P}$ is just the union of these pyramids. Let $j\in J$. There are exactly $[W_J:W_{J\setminus\{j\}}]$ pyramids of the form $h(w,j)$, with $w\in W_J/W_{J\setminus\{j\}}$. All of these pyramids have equal $|J|$-dimensional volume, which is given by
\begin{equation*}
    \text{Vol}_{|J|}\left(h(\text{id},j)\right)=\dfrac{1}{|J|} d_j\,\text{Vol}_{|J|}\left(F(\text{id},j)\right)
    =\dfrac{1}{|J|} d_j\, V_{J\setminus\{j\}}^\Phi(\lambda),
\end{equation*}
where $d_j$ denotes the distance from $\xi_{J^c}$ to the hyperplane $L_{J\setminus\{j\}}+\lambda$ of $U$. By definition, $\nu_j\in L_J$ is orthogonal to $L_{J\setminus\{j\}}$. Therefore,
\begin{equation*}
    d_j=\dfrac{(\lambda-\xi_{J^c},\nu_j)}{\Vert\nu_j\Vert}=\dfrac{(\lambda,\nu_j)}{\Vert\nu_j\Vert},
\end{equation*}
since $(\xi_{J^c},\nu_j)=0$. 
This gives the desired result, assuming $m_i>0$ for all $i\in J$. 

Finally if some $m_i$ are zero, we can still divide $\mathcal{P}$ according to its facets. 
They are contained in $\mathcal{F}$ but the containment may be proper, see Remark \ref{rem:facets}.
Suppose ${F(w,j)\in\mathcal{F}}$ is not a facet. 
Then $F(w,j)$ is still a face of $\mathcal{P}$, but has dimension $<|J|-1$. 
Then the corresponding (degenerated) pyramid ${h(w,j)}$ has zero $|J|$-dimensional volume, so these extra faces are not a problem.
\end{proof}

\begin{remark}\label{rmk: volumes are polynomials}
Let ${\mathbf{m}=(m_i)_{i\in I_n}}$ be a $n$-tuple of non-negative real numbers and $J\subset I_n$.
Let $V_J^\Phi(\mathbf{m})\coloneqq V_J^\Phi(m_1\varpi_1^\vee+\cdots+m_1\varpi_n^\vee)$.
Lemma \ref{lem: generalized pyramid formula} implies that $V_J^\Phi(\mathbf{m})$ is in fact a homogeneous polynomial of degree $|J|$ in the variables $\{m_j \mid j\in J\}$. 
Indeed, note that $$(\lambda,\nu_j)=\sum_{i\in J}m_i(\varpi_i^\vee,\nu_j),$$ for every $j\in J$, that is, $(\lambda,\nu_j)$ is a homogeneous polynomial of degree $1$ in the desired variables. 
Since $V_\emptyset^\Phi=1$, the result follows by induction.
\end{remark}

From now on, $V_J^\Phi$ will mean the corresponding polynomial in $\mathbb{R}[m_1,m_2,\ldots,m_n]$. 


\begin{lemma}\label{lem:squarefree-coefficients}
Let $J,K \subset I_n$ and define ${\displaystyle m_J\coloneqq\prod_{j\in J}m_j\in\mathbb{R}[m_1,\ldots,m_n]}$.
Let $c_{J,K}^{\Phi}$ be the coefficient of $m_J$ in ${V_K^\Phi(\mathbf{m})}$ and $c_{J}^{\Phi} = c_{J,J}^{\Phi}$. 
\begin{enumerate}[(i)]
    \item If ${K\neq J}$ then $c_{J,K}^{\Phi}=0$.
    \item $c_J^{\Phi}>0$.  
\end{enumerate}
\end{lemma}

\begin{proof} If $J=\emptyset$ both statements are trivial. Thus we can assume that $J\neq\emptyset$. 

\begin{enumerate}[(i)]
    \item   We recall that  ${V_K^\Phi(\mathbf{m})}$ is a homogeneous polynomial in the variables $\{m_k\}_{k\in K}$. Suppose that $c_{J,K}^{\Phi}\neq 0$.  It follows that ${J\subset K}$. On the other hand, $m_J$ has degree ${|J|}$, so that ${|J|=|K|}$. Therefore, $J=K$ which contradicts our hypothesis. We conclude that $c_{J,K}^{\Phi}=0$.  
    \item   The coefficient of $m_J$ in $(\lambda, \nu_j)\,V_{J\setminus\{j\}}^\Phi(\lambda)$ is $(\varpi_j^\vee,\nu_j)\,c_{J\setminus\{j\}}^\Phi$. Thus,  \eqref{eq: generalized pyramid formula} gives
\begin{equation*}
    c_J^\Phi=\dfrac{1}{|J|}\sum_{j\in J}[W_J:W_{J\setminus\{j\}}]\,(\varpi_j^\vee,\nu_j)\,c_{J\setminus\{j\}}^\Phi.
\end{equation*}
By induction, we can assume $c_{J\setminus\{j\}}^\Phi>0$. Thus, the result follows by  Lemma \ref{lem: pyramidprev}\eqref{trestres}. 
\end{enumerate}
\end{proof}

\begin{cor}
\label{cor:independent}
The polynomials $V_J^\Phi(\mathbf{m})$, with ${J\subset I_n}$, are linearly independent.
\end{cor}

\begin{proof} 
The result follows by a direct application of Lemma \ref{lem:squarefree-coefficients}. 
\end{proof}

\begin{remark}\label{rmk: postinkov}
So far we have used the Euclidean volume on $V$.
To compare our results to the volume formulas of Postnikov there is a scalar factor of $ \sqrt{n+1} $ missing for type $ A_n $.
This is because his formulas are stated with the volume relative to the root lattice, that is, it is scaled so that the fundamental parallelepiped spanned by the simple roots has volume 1, but the Euclidean volume is $ \sqrt{n+1} $ by Equation \eqref{eq:vol_fundamental_root}.
Our variables $m_1,\cdots,m_n$ correspond to the variables $u_1,\cdots, u_n$ in \cite[Section 16]{postnikov2009permutohedra}.
\end{remark}

\section{Geometric Formula}

The purpose of this section is to prove Theorem \ref{geofor}. For the reader's convenience, we state the theorem again.
\begin{theorem}\label{thm:geometric-formula}
For every root system $\Phi$, there are unique ${\mu_J^{\Phi}\in\mathbb{R}}$ such that for any dominant coweight $\lambda$,
\begin{equation}\label{geofordef}
    |\leq\theta(\lambda)|=\sum_{J\subset I_n}\mu_J^{\Phi}V_J^{\Phi}(\lambda).
\end{equation}
\end{theorem}
This implies that if $\Phi$ has rank $n$ and ${\lambda=(m_i)_{i\in I_n}}$ in the fundamental coweight basis, then ${|\leq\theta(\lambda)|}$ is a polynomial of degree $n$ in the ${m_1,\ldots,m_n}$. Taking the sum over a fixed rank ${|J|=d}$ gives the degree $d$ part of the polynomial.
We call the coefficients $\mu_J^\Phi$ the \textit{geometric coefficients}.
First we review some more concepts from discrete geometry.

\subsection{Transverse cones}
Let $\Gamma\subset V$ be a lattice.
A polytope $\mathsf{P}$ in $V$ is called a \emph{lattice} polytope if all its vertices are in $\Gamma$; it is called \emph{rational} if an integer dilation of it is a lattice polytope.
A pointed cone $\mathsf{C}$ is rational if its vertex is a lattice point and every ray (1-dimensional face) contains a lattice point.

Let $L\subset V$ be a linear subspace.
If $L$ has a basis consisting of elements of $\Gamma$, then $\Gamma/L$ is also a lattice in the quotient $V/L$ by \cite[Corollary 10.3]{barvinok2008integer}.
Since $ V $ has an inner product we can define $ L^{\perp}. $ There is a canonical isomorphism between $ V/L $ and $ L^{\perp} $ under which the map $ V \to V/L $ corresponds to the orthogonal projection to $ L^{\perp} $.
Notice that in general the set $\Gamma \cap L^{\perp}$ is different from the projection of $\Gamma$ into $L^{\perp}$ as can be seen in Figure \ref{fig:enter-label}.
For each facet $\mathsf{G}$ of $\mathsf{P}$ we define an outer normal as an element $u_G\in V$ such that every point in the polytope satisfies the inequality $\langle u_{\mathsf{G}}, v\rangle \leq b_{\mathsf{G}}$, for some constant $b_\mathsf{G}$, with equality iff $v\in \mathsf{G}$.
If $\mathsf{P}$ is a lattice polytope, then $u_{\mathsf{G}}$ can be chosen to be a lattice vector. Furthermore, if $\mathsf{P}$ is full dimensional then $u_\mathsf{G}$ is unique up to positive scalar. 

\begin{figure}[ht]
    \centering
    \includegraphics[scale=2]{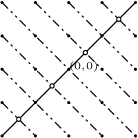}
    \caption{}
    \label{fig:enter-label}
\end{figure}

\begin{definition}
\label{def:conos}
Let $\Gamma \subset V$ a lattice and let $ \mathsf{P} $ be a full dimensional  lattice polytope.
For every face $ \mathsf{F} \subset \mathsf{P} $ let $ H $ be its affine span, $ L $ the corresponding linear subspace.
We define the following cones:

\begin{itemize}

    \item The \textbf{feasible cone} $\mathsf{f}(\mathsf{F},\mathsf{P}) = \{ v\in V ~:~ \exists\  \epsilon>0 \text{ such that } x + \epsilon v \in P \}$, where $x$ is a point in the relative interior of $\mathsf{F}$, i.e., a point in $\mathsf{F}$ that does not belong to any proper face of it.
    The feasible cone is independent of the interior point $x$.
    \item The \textbf{supporting cone} $\mathsf{s}(\mathsf{F},\mathsf{P}) := H + \mathsf{f}(\mathsf{F},\mathsf{P})$. By definition this is a translation of the feasible cone.
    \item The \textbf{transverse cone} $ \mathsf{t}( \mathsf{F} , \mathsf{P} ) = \mathsf{s}(\mathsf{F},\mathsf{P})/L \subset V/L = L^\perp \subset V$.
    \item The \textbf{normal cone} $ \mathsf{n}( \mathsf{F} , \mathsf{P} ) = \cone \{ u_\mathsf{G}~:~\mathsf{G}\text{ is a facet such that } \mathsf{F}\subset \mathsf{G}\}$.
\end{itemize}
For any of these cones, we simply say ``the cone of $\mathsf{F}$'' when $\mathsf{P}$ is clear from context.
\end{definition}

The first three cones are visibly related.
For any (possibly non-pointed) cone $\mathsf{C}$ that includes the origin, we define its \emph{polar}.
\[
\mathsf{C}^\circ = \{v\in V~:~( v,w ) \leq 0, \forall w\in \mathsf{C}\}.
\]
The normal cone is the polar of the feasible cone, see \cite[Theorem 6.46]{rockafellar2009variational} (in that source the feasible cone is called the tangent cone).
This set of definitions begs for an example.

\begin{example}
Consider the lattice $\mathbb{Z}^2$ in $\mathbb{R}^2$ and the lattice polygon given by
\[
\mathsf{P}
=
\conv
\begin{bmatrix}
-3 & -4 & -3 & 0 \\
1 & 3 & 4 & 2
\end{bmatrix}
\]

Let's first analyze the vertex $v=(-3,1)$.
Figure \ref{fig:cone_vertex} illustrates the four cones.
The supporting cone is the cone with vertex $v$ and emanating towards the polytope.
The feasible cone is the translation of this cone that places the vertex in the origin.
The linear space spanned by $v$ is trivial so $V/L = V$ and thus the transverse cone agrees with the supporting cone.
The vertex $v$ is contained in two facets, which are edges in this case.
The vectors $(-2,-1)$ and $(1,-3)$ are outer normals of the edges $\{(-3,1),(-4,3)\}$ and $\{(-3,1), (0,2)\}$, respectively.
These two outer normals (or rather a dilation) are depicted in Figure \ref{fig:cone_vertex} as dashed arrows perpendicular to the faces.
The normal cone is the cone with vertex at the origin spanned by these two vectors.

\label{ex:conos}
\begin{figure}[ht]
    \centering
    \begin{minipage}{0.45\textwidth}
        \centering
        \includegraphics[scale=1.2]{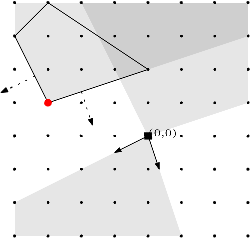} 
        \caption{The four cones in the case of the vertex (-3,1) of the lattice polygon. Transverse and supporting are the same in this case.}
	\label{fig:cone_vertex}
    \end{minipage}\hfill
    \begin{minipage}{0.45\textwidth}
        \centering
        \includegraphics[scale=1.2]{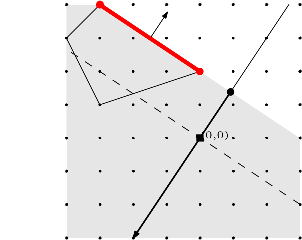} 
        \caption{The transverse cone of the edge $\{(-3,4),(0,2)\}$. Its direction is the negative of the outer normal.}
	\label{fig:cone_general}
    \end{minipage}
\end{figure}

Now let's analyze the edge $\{(-3,4),(0,2)\}$.
In Figure \ref{fig:cone_general} the supporting cone and the transverse cone are shown.
First note that the linear subspace $L$ corresponding to this edge is generated by the vector $(3,-2)$ and it is shown as a dashed line through the origin.
The feasible cone (not depicted) is everything below that line.
The supporting cone is everything below the line spanned by the edge (note how it is not pointed).
The transverse cone is the projection to the orthogonal complement of $L$.
The outer normal can be chosen to be $(2,3)$, and it is depicted as an arrow perpendicular to the edge. The normal cone is generated by $(2,3)$, which is an outer normal of the edge.
Notice that the normal cone is a ray that is polar to the feasible cone.
\end{example}

We can now explain a less trivial example, the transverse cones for orbit polytopes.
We use the notation $\bar{v}$ to denote the image of $V$ in a quotient map.

\begin{lemma}
\label{lem:tcones}
Let $\Phi\subset V$ be a root system and let $\lambda\in C^+$ be generic.
Let $\mathsf{F}_J(\lambda) \subset \mathsf{P}^\Phi(\lambda)$ be as in Definition \ref{def:face_representative}.
Then 
\begin{equation}
\label{eq:tcone_orbit}
\mathsf{t}(\mathsf{F}_J(\lambda),\mathsf{P}^\Phi(\lambda)) = \bar{\lambda} + \cone\{\bar{-\alpha^\vee_i}\,:\,i \notin J\} \subset V/L_J,
\end{equation}
where $L_J$ is defined in Equation \eqref{eq:face_linear}.
\end{lemma}

\begin{proof}

We put $\mathsf{F}_J=\mathsf{F}_J(\lambda)$ and $\mathsf{P}=\mathsf{P}^\Phi(\lambda)$ for simplicity.
The linear space generated by $F_J$ is $L_J$.
Note that the facets of $\mathsf{P}$ containing $\mathsf{F}_J$, are precisely $\mathsf{F}_{I_n\setminus\{i\}}$ with $i\notin J$, since $\lambda$ is generic.
It follows that the normal cone of $\mathsf{F}_J$ is generated by $\{\varpi_i^\vee\,:\, i\notin J\}$. 
As the feasible cone is polar to the normal cone, we have that 
\begin{align*}
    \mathsf{f}(\mathsf{F}_J,\mathsf{P}) 
    &= \{v\,:\, (v,\varpi^\vee_i)\leq 0 \text{ for all } i\notin J\}\\
    &= \spa\{\alpha^\vee_j\,:\,j\in J\}+\cone\{-\alpha^\vee_i\,:\,i\notin J\}.
\end{align*}
By definition, $\mathsf{s}(\mathsf{F}_J,\mathsf{P})=\lambda+\mathsf{f}(\mathsf{F}_J,\mathsf{P})$ and $\mathsf{t}(\mathsf{F}_J,\mathsf{P})=\overline{\mathsf{s}(\mathsf{F}_J,\mathsf{P})}$.
Equation \eqref{eq:tcone_orbit} follows, since $\spa\{\alpha^\vee_j\,:\,j\in J\}=L_J$. 
\end{proof}


\subsection{Euler-Maclaurin formulas}
The following is the  \emph{Euler-Maclaurin formula} developed by Berline and Vergne \cite{berline2007local} (see also \cite[Chapters 19-20]{barvinok2008integer} for an exposition).
There exist a function $\nu$ on pointed rational cones such that the following is true for all lattice polytopes $ \mathsf{P}. $
\begin{equation}\label{eq:local-formula}
|\mathsf{P}\cap \Gamma| = \sum_{\mathsf{F}\subseteq \mathsf{P}} \nu\left( \mathsf{t}( \mathsf{F} , \mathsf{P}) \right)\text{relVol} (\mathsf{F}),
\end{equation}
where the sum is indexed over all nonempty faces of $\mathsf{P}$. 
The relative volume $\mathrm{relVol}(\mathsf{F})$ of a face is the volume form on its affine span $H$ normalized with respect to the lattice $\Gamma\cap L$, where $L$ is the linear subspace parallel to $H$. More precisely,
\begin{equation}
\label{eq:rel_vol}
\relvol(\mathsf{F}) = \frac{\vol(\mathsf{F})}{\det(\Gamma\cap L)}.
\end{equation}

\begin{remark}
To be more precise, Berline and Vergne's main construction in \cite{berline2007local} is a function $\mu$ that maps pointed rational cones to meromorphic functions \cite[Section 4]{berline2007local}.
In this paper we only use the function $\nu$ which is $\mu$ evaluated at zero \cite[Definition 25]{berline2007local}, and then Equation \eqref{eq:local-formula} is equivalent to \cite[Theorem 26]{berline2007local} when the function $h$ is the constant function equal to $1$.

Alternatively, we are using the construction of Pommersheim and Thomas in \cite{pommersheim2004cycles} in the case where the complement map is given by an inner product, see \cite[Corollary 1 (iv)]{pommersheim2004cycles}.
Both constructions are involved and only in few cases the actual value of the $\nu$ function is known.
See \cite{fischer2022algebraic} for a purely combinatorial construction of this function.
\end{remark}

We remark that for a single polytope $P$, it is obvious that there will be a formula of the sort.
The interesting part of Berline-Vergne's  theorem is that the 
 $\nu$ function satisfies Equation \eqref{eq:local-formula} for all lattice polytopes simultaneously and has certain local properties (see Lemma \ref{lem:orbit_nu}).

\begin{example}[Pick's Theorem]\label{Pick}
Let $ \mathbb{Z}^2 \subset \mathbb{R}^2 $ be the ambient lattice and space and $ \mathsf{P} = \conv\{(1,0),(3,0),(4,3),(4,4),(3,4),(0,1)\} $ be the lattice polygon depicted in Figure \ref{fig:pick}.
Equation \eqref{eq:local-formula} states that the total 15 lattice points in the polygon can be accounted in the following way.
\begin{description}
\item[Dimension 2] The contribution to the sum of the whole polygon is its relative volume, which in this case is simply its area: $ 19/2 $. This is because the $ \nu $ function applied to the transverse cone of the whole polygon is 1 (one can prove this using that the $\nu$ function applied to a singleton gives $1$). 
\item[Dimension 1] For each edge the $ \nu $ value of the corresponding transverse cone is always $ 1/2 $. Additionally, we compute the relative volume. For example in the edge $(0,1)-(3,4)$ we normalize the volume in the spanning subspace so that the fundamental parallelepiped $(0,0)-(1,1)$ has area one. 
This segment then has relative volume of $3$.
The total contribution of the edges is equal to $ (3+1+1+1+2+1)/2 = 9/2$
\item[Dimension 0] For each vertex the relative volume is equal to one. So we are simply adding the $ \nu $ values. In this case, they are not all equal, but  there is a property of   $ \nu $ called \emph{valuation} that says that they add up to $ 1 $ when summed over the transverse cones of vertices.
\end{description}
Extrapolating this example we get that for any lattice polygon $ \mathsf{P} $ we have the formula
\begin{equation}\label{eq: pick}
|\mathsf{P} \cap \mathbb{Z}^{2}| = \text{Area}(\mathsf{P}) + \frac{1}{2}\text{Boundary Points}(\mathsf{P}) + 1,
\end{equation}
which is known as Pick's formula.

\begin{figure}[ht]
\centering
\includegraphics{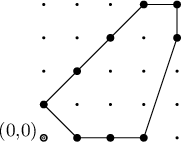} 
\caption{A lattice polygon with 15 lattice points.}
\label{fig:pick}
\end{figure}
\end{example}

Now we analyze these concepts in the case of generic orbit polytopes.
Recall that we know their face structure by Corollary \ref{cor:generic}, in particular their $W_f$-orbits of faces are in bijection with subsets $ J \subset I_n $.
Also, we introduce a translation of orbit polytopes
\begin{equation}
    \label{eq:Q_def}
\mathsf{Q}^\Phi(\lambda) = \mathsf{P}^\Phi(\lambda) - \lambda.
\end{equation}
The polytope $\mathsf{Q}^\Phi$ is always a $\mathbb{Z}\Phi^\vee$-lattice polytope, which simplifies some considerations below.
Its faces are $\mathsf{G}_J(\lambda,w) := w\mathsf{F}_J(\lambda)-\lambda$ for all pairs of $w\in W_f, J\subset I_n$.
We define $\mathsf{G}_J(\lambda):=\mathsf{G}_J(\lambda,\mathrm{id}).$

\begin{lemma}
\label{lem:orbit_nu}
Let $ \Phi $ be a root system of rank $n$, $J\subset I_n$, $ \lambda $ a \textbf{generic} element of $(\Lambda^\vee)^+$ and $\mathsf{Q}^\Phi(\lambda)$ defined as in Equation \eqref{eq:Q_def}.
Then
\begin{enumerate}
    \item The $\nu$ value of the transverse cone of $\mathsf{G}_J(\lambda)$ in $\mathsf{Q}^\Phi(\lambda)$  is independent of $\lambda$. 
    
    

    \item The $\nu$ value of the transverse cones of  $\mathsf{G}_J(\lambda)$ and $\mathsf{G}_J(\lambda,w)$ are equal for all $w\in W_f$. 
    
    
    \item For $w\in W_f$ we have that $\vol(\mathsf{G}_J(\lambda))=\vol(\mathsf{G}_J(\lambda,w))$. Furthermore, $\relvol(\mathsf{G}_J(\lambda))=\relvol(\mathsf{G}_J(\lambda,w))$.

\end{enumerate}
\end{lemma}

\begin{proof}
We need to use two properties of $ \nu $.
These two results are part of the content of \cite[Proposition 14]{berline2007local}.
The following operations do not change the $\nu$ value of a transverse cone.
\begin{itemize}
    \item [i] Applying a lattice-preserving orthogonal transformation.
    \item [ii] Translating by a lattice element.
\end{itemize}

\begin{enumerate}

\item  Since $\lambda$ is generic, by applying Lemma \ref{lem:tcones} we obtain 
\begin{equation*}
    \mathsf{t}( \mathsf{G}_J(\lambda) , \mathsf{Q}^\Phi(\lambda)) = \mathsf{t}( \mathsf{F}_J(\lambda) , \mathsf{P}^\Phi(\lambda)) -\overline{\lambda}
    = \cone\{\overline{-\alpha_j^\vee}~:~j\notin J\}.
\end{equation*}
Clearly, the right-hand side of the previous equality does not depend on $\lambda$.
Therefore, the $\nu$ value does not depend on $\lambda$ as well. 



\item  We can rewrite 
\begin{equation}
\label{eq:F_J}
\mathsf{G}_J(\lambda,w)=w\mathsf{G}_J(\lambda)-(\lambda-w\lambda).
\end{equation}

The facets containing $\mathsf{G}_J(\lambda)$ are $\mathsf{G}_{I_n\setminus\{i\}}(\lambda)$ with $i\notin J$.
Then the facets containing $\mathsf{G}_J(\lambda,w)$ are $\mathsf{G}_{I_n\setminus\{i\}}(\lambda,w) = w\mathsf{G}_{I_n\setminus\{i\}}(\lambda)-(\lambda-w\lambda)$ for $i\notin J$, the equality by Equation \eqref{eq:F_J}.
Notice that $w\mathsf{G}_{I_n\setminus\{i\}}(\lambda)-(\lambda-w\lambda)$ is simply a translation of $w\mathsf{G}_{I_n\setminus\{i\}}(\lambda)$ and as such it has the same outer normal.

Since $w$ as a linear transformation preserves the inner product, then the set $\{w\varpi_i\,:\, i\notin J\}=w\{\varpi_i\,:\, i\notin J\}$ are the outer normals of the facets containing $\mathsf{G}_J(\lambda,w)$.
This shows that 
\begin{equation}
\label{eq:moving_feasible}
\mathsf{n}(\mathsf{G}_J(\lambda,w),\mathsf{Q}^\Phi(\lambda)) = w\cdot \mathsf{n}(\mathsf{G}_J(\lambda),\mathsf{Q}^\Phi(\lambda)).
\end{equation}

This implies that
\small
$$\mathsf{s}(\mathsf{G}_J(\lambda,w),\mathsf{Q}^\Phi(\lambda)) - \left(w\lambda-\lambda\right) = \mathsf{f}(\mathsf{G}_J(\lambda,w),\mathsf{Q}^\Phi(\lambda)) = w \mathsf{f}(\mathsf{G}_J(\lambda),\mathsf{Q}^\Phi(\lambda))= w \mathsf{s}(\mathsf{G}_J(\lambda),\mathsf{Q}^\Phi(\lambda)).$$
\normalsize

In the first equality, we subtract only a vector of the supporting cone because the linear span of the face $\mathsf{G}_J(\lambda,w)$ is contained in the feasible cone.
The inner equality follows by taking the polar on both sides of Equation \eqref{eq:moving_feasible}. Indeed notice that $\langle v,C\rangle \leq 0$ if and only if $\langle wv,wC\rangle \leq 0$, so the polar is also compatible with multiplying by $w$. The last equality follows from the first one.

Rearranging we obtain
\begin{equation}\label{eq:supporting}
\mathsf{s}(\mathsf{G}_J(\lambda,w),\mathsf{Q}^\Phi(\lambda)) = w \mathsf{s}(\mathsf{G}_J(\lambda),\mathsf{Q}^\Phi(\lambda))+\left(w\lambda -\lambda\right).
\end{equation}

Notice that the linear span of the face $\mathsf{G}_J(\lambda,w)$ is  $wL_J$.
Also, as $w$ is an orthogonal transformation, $wL_J^\perp=(wL_J)^\perp$.
So we can project both sides of Equation \eqref{eq:supporting} to arrive at

\begin{equation*}
    \mathsf{t}( \mathsf{G}_J(\lambda,w) , \mathsf{Q}^\Phi(\lambda)) = w\mathsf{t}( \mathsf{G}_J(\lambda) , \mathsf{Q}^\Phi(\lambda))+\pi({w\lambda-\lambda})\subset wL_J^\perp.
\end{equation*}

Here $\pi({w\lambda-\lambda})$ is the image of $w\lambda-\lambda$ under the orthogonal projection to $wL_J^\perp$.
The second vector is an element of the projected lattice $\mathbb{Z}\Phi^\vee$ in $wL_J^\perp$ so $w\mathsf{t}( \mathsf{G}_J(\lambda) , \mathsf{Q}^\Phi(\lambda))-\pi({w\lambda-\lambda})$ has the same $\nu$-value as $w\mathsf{t}( \mathsf{G}_J(\lambda), \mathsf{Q}^\Phi(\lambda))$.
The linear map $w$ is a lattice-preserving orthogonal transformation, so by the properties mentioned above
$w\mathsf{t}( \mathsf{G}_J(\lambda))$ and $\mathsf{t}( \mathsf{G}_J(\lambda))$ have the same $\nu$-value.
Putting everything together, both cones $\mathsf{t}( \mathsf{G}_J(\lambda))$ and $\mathsf{t}( \mathsf{G}_J(\lambda,w))$ have the same $\nu$-value.

\item The two faces differ by applying $w$ and translating. Since $ W_f $ is a subgroup of the orthogonal group of $V$, the volume does not change by translations and multiplications by elements in $W_f$. This proves the first claim. 

For the second claim, we notice that the restrictions of the lattice to both faces are isomorphic by the orthogonal transformation $w$, so both fundamental parallelepipeds have the same volume and hence the relative volume of both faces agree. 
\end{enumerate}
\end{proof}

As a consequence of the last item of Lemma \ref{lem:orbit_nu} we have that the quotient between the relative volume and the volume does not depend on $w$.
It doesn't depend on $\lambda$ either, as $L_J$ is the linear subspace spanned by it.
We define
\begin{equation}
\label{eq:v_J}
v_J := \frac{\relvol(\mathsf{G}_J(\lambda,w))}{\vol(\mathsf{G}_J(\lambda,w))}
= \frac{1}{\det(\mathbb{Z}\Phi^\vee \cap L_J)},
\end{equation}
where the second equation follows from Equation \eqref{eq:rel_vol}.

\subsection{Proof of the geometric formula}

We first prove the existence of such a formula for $\lambda$ generic.
\begin{prop}\label{prop:geometric-formula-generic}
For every root system $\Phi$, there exists ${\mu_J^{\Phi}\in\mathbb{R}}$ such that for any \textbf{generic} dominant coweight $\lambda$,
\begin{equation}
|\leq\theta(\lambda)|=\sum_{J\subset I_n}\mu_J^{\Phi}V_J^{\Phi}(\lambda).
\end{equation}
\end{prop}

\begin{proof}

Using the Lattice Formula, Equation \eqref{firstlatdef}, we have 
\begin{equation}
\label{eq:complete_lattic_formula}
|\leq\theta(\lambda)|=|W_f| \ |\mathsf{P}^\Phi(\lambda)\cap (\lambda+\mathbb{Z}\Phi^\vee)|=|W_f| \ |\big(\mathsf{P}^\Phi(\lambda) - \lambda\big) \cap \mathbb{Z}\Phi^\vee|.
\end{equation}

By Proposition \ref{prop:face_lattice}, the set of vertices of $\mathsf{P}^\Phi(\lambda)$ is
$W_f\cdot \lambda.$ By definition of a coweight and that of the action of $W_f$ on $V$, $\lambda-w(\lambda) \in \mathbb{Z}\Phi^\vee$, for $w\in W_f.$ Thus
the polytope $ \mathsf{Q}^\Phi(\lambda) = \mathsf{P}^\Phi(\lambda) - \lambda$ is a lattice polytope with respect to the lattice $ \mathbb{Z}\Phi^\vee $. 
We use Berline-Vergne formula, Equation \eqref{eq:local-formula}, to obtain
\begin{equation}\label{eq-step2}
\ |\mathsf{Q}^\Phi(\lambda) \cap \mathbb{Z}\Phi^\vee|= \sum_{\mathsf{F} \subseteq\mathsf{Q}^\Phi(\lambda)} \nu\left( \mathsf{t}( \mathsf{F} , \mathsf{Q}^\Phi(\lambda))\right) \text{relVol} (\mathsf{F}).
\end{equation}

 For $J\subset I_n$, let $\mathcal{F}_J=\{F_J'(\lambda,w)~:~w\in W_f\}$. 
By part 2 and 3 of Lemma \ref{lem:orbit_nu} and Corollary \ref{cor:generic}, we get
\begin{align*}
    |\mathsf{Q}^\Phi(\lambda) \cap \mathbb{Z}\Phi^\vee|&=\sum_{J\subset I_n}\sum_{\mathsf{F}\in\mathcal{F}_J} \nu\left( \mathsf{t}( \mathsf{F} , \mathsf{Q}^\Phi(\lambda))\right) \text{relVol} (\mathsf{F}) \\
    &=\sum_{J\subset I_n} [W_f:W_J] \, \nu\left( \mathsf{t}( \mathsf{G}_J(\lambda) , \mathsf{Q}^\Phi(\lambda))\right) \text{relVol} (\mathsf{G}_J(\lambda)).
\end{align*}
Since $\mathsf{Q}^\Phi(\lambda)$ is just a translation of $\mathsf{P}^\Phi(\lambda)$, it is clear that $V_J^\Phi(\lambda)=\vol(\mathsf{G}_J(\lambda))$.
Therefore, substituting in \eqref{eq:complete_lattic_formula} we get
\begin{equation}\label{eq-final}
|\leq\theta(\lambda)| = \sum_{J\subseteq I_n}\mu_J^{\Phi}V_J^{\Phi}(\lambda),
\end{equation}
where 
\begin{equation}
\label{eq:geometric_coefficient}
\mu_J^{\Phi} := |W_f|\frac{|W_f|}{|W_J|}\nu\left( \mathsf{t}( \mathsf{G}_J(\lambda) , \mathsf{Q}^\Phi(\lambda)) \right) v_J.
\end{equation}
Finally, Lemma \ref{lem:orbit_nu} part 1 implies that $\mu_J^{\Phi} $ does not depend on the choice of $\lambda$.
\end{proof}

A function $ g:\mathbb{Z}_{\geq0}^{n} \to \mathbb{R} $ is called a \emph{multivariate quasi-polynomial} if there exists a finite-index lattice $\mathcal{L}\subset\mathbb{Z}^n$ (i.e. $\{\mathbf{u}_i\}_{i\in I}=\mathbb{Z}^n/\mathcal{L}$, with $I$ finite), and polynomials $p_i\in\mathbb{R}[x_1,\ldots, x_n]$ such that if $\mathbf{m}:=(m_1, \ldots, m_n)\in \mathbb{Z}^n,$ for all $i\in I$ we have
$$g(\mathbf{m})=p_i(\mathbf{m}), \text{ if } \overline{\mathbf{m}}=\mathbf{u}_i.$$

\begin{prop}
\label{prop:quasi-polynomiality}
For every dominant coweight $\lambda = \sum_i m_i\varpi_i^\vee$, we have that $|\leq \theta(\lambda)|$ is a quasi-polynomial in $m_1,\dots,m_n$.
\end{prop}

\begin{proof}
By the Lattice Formula (Theorem \ref{thm:first-lattice-formula}) it is enough to prove the quasi-polynomiality of
\begin{equation}\label{eq: quasi}
    |\mathsf{P}^\Phi(\lambda)\cap (\lambda+\mathbb{Z}\Phi^\vee )| = |\left(\mathsf{P}^\Phi(\lambda) - \lambda\right)\cap\mathbb{Z}\Phi^\vee|.
\end{equation}

Recall that the Minkowski sum of two polytopes $\mathsf{P}$ and $\mathsf{Q}$ is the polytope $\mathsf{P}+\mathsf{Q}=\conv\{p+q\,:\, p\in\mathsf{P},\, q\in\mathsf{Q}\}$.
Since $\lambda = \sum_i m_i\varpi_i^\vee$ we have the following equality of polytopes (see \cite[Proposition 6.4]{submodular}):
\begin{align}
\mathsf{P}^\Phi(\lambda) &= m_1\mathsf{P}^\Phi(\varpi_1^\vee)+m_2\mathsf{P}^\Phi(\varpi_2^\vee)+\dots+m_n\mathsf{P}^\Phi(\varpi_n^\vee),\\
\mathsf{Q}^\Phi(\lambda) &= m_1\mathsf{Q}^\Phi(\varpi_1^\vee)+m_2\mathsf{Q}^\Phi(\varpi_2^\vee)+\dots+m_n\mathsf{Q}^\Phi(\varpi_n^\vee) \label{eq:minkowski_sum},
\end{align}
where $\mathsf{Q}^\Phi$ is defined in \eqref{eq:Q_def}.
In Equation \eqref{eq:minkowski_sum} every polytope on the right-hand side is  $\mathbb{Z}\Phi^\vee$-rational since the index of connection is finite. 
 Following McMullen \cite[Theorem 7]{mcmullen1978lattice} we have that the number of lattice points in an integer Minkowski sum of rational polytopes is a quasi-polynomial in the dilation factors.
This implies that the right-hand side of $\eqref{eq: quasi}$ is a quasi-polynomial in $m_1,\ldots,m_n$.
\end{proof}

We will now prove that $|\leq \theta(\lambda)|$ is an honest-to-god polynomial.

\begin{lemma}
\label{lem:aux}
Let $h$ be a quasipolynomial in $n\geq 1$ variables such that $h(\mathbf{m})=0$ whenever $\mathbf{m} \in \mathbb{Z}_{>0}^n$.
Then $h$ is identically zero.
Consequently, if two quasipolynomials agree in $\mathbb{Z}_{>0}^n$,  they agree everywhere.
\end{lemma}



\begin{proof}
    Let $\mathcal{L}\subset\mathbb{Z}^n$ be a finite-index lattice.
    It is enough to show that for any $\mathbf{u} \in \mathbb{Z}^n_{\geq 0}$ and for any polynomial  $p\in \mathbb{R}[x_1,\ldots , x_n]$   such that $p(\mathbf{m}) =0$, for all $ \mathbf{m} \in  (\mathbf{u} + \mathcal{L})\cap \mathbb{Z}^n_{> 0}$, we have that $p$ is the zero polynomial. We will prove this in two steps.

\begin{claim}If  $\mathbf{u} \in \mathbb{Z}^n_{\geq 0},$  $p\in \mathbb{R}[x_1,\ldots , x_n]$ are  such that $p(\mathbf{m}) =0$ for all $ \mathbf{m} \in  (\mathbf{u} + \mathcal{L})\cap \mathbb{Z}^n_{> 0}$, then $p$ is  zero in $\mathbb{Z}^n_{> 0}$. 
\end{claim}
  \begin{proof}
    We first prove the  case  when $\mathbf{u}=\mathbf{0}$.
    If $p$ vanishes on $\mathcal{L}\cap \mathbb{Z}^n_{> 0} $, then for every $\mathbf{v} =(v_1,v_2,\ldots ,v_n)\in\mathcal{L}\cap \mathbb{Z}^n_{> 0}$ the univariate polynomial
    $p(tv_1,tv_2,\cdots,tv_n)\in \mathbb{R}[t]$ vanishes for every $m\in \mathbb{Z}_{>0}$, hence it is the zero polynomial.
    This means that $p$ vanishes on all lines through the origin containing an element of $\mathcal{L}$, i.e. $\mathcal{L}$-\emph{rational lines}. On the other hand, any element of  
    $ \mathbb{Z}^n_{>0}$ belongs to some $\mathcal{L}$-rational line since $\mathcal{L}$ is a finite-index lattice. Summing up, $p$ vanishes on $\mathbb{Z}^n_{> 0} $.   When $\mathbf{u} \neq \mathbf{0}$, we can conclude with a similar argument by doing a change of coordinates. This ends the proof of the claim.
    \end{proof}
    
Let us return to the proof of the lemma. It remains to show that a polynomial that vanishes in $\mathbb{Z}^n_{> 0} $ must be the zero polynomial. We proceed by induction on $n$. If $n=1$, then $p$ has infinitely many zeros so it is the zero polynomial.  For $n>1$ let us write
\begin{equation} \label{eq: h in terms of xn}
    p= \displaystyle  \sum_{i\geq0} p_i(x_1,\ldots , x_{n-1}) x_{n}^i,
\end{equation}
where $p_i \in \mathbb{R}[x_1, \dots,x_{n-1}]$. For each fixed $(n-1)$-tuple $(a_1,\ldots , a_{n-1})\in \mathbb{Z}^{n-1}_{>0}$ the polynomial $p(a_1,\ldots ,a_{n-1},x_n)\in \mathbb{R}[x_n]$ vanishes for all $x_n\in \mathbb{Z}_{>0}$, and therefore, it is the zero polynomial.  It follows that  $p_i(a_1,\ldots , a_{n-1})=0 $ for all  $(a_1,\ldots , a_{n-1})\in \mathbb{Z}^{n-1}_{>0}$. By our inductive hipothesis we conclude that $p_i(x_1,\ldots , x_{n-1})$ is the zero polynomial for all $i\geq 0$. By substituting in \eqref{eq: h in terms of xn} we get $p=0$.

 Finally, the second claim in the lemma follows by considering the difference of the two quasipolynomials.
 \end{proof}   

\begin{proof}[Proof of Theorem \ref{thm:geometric-formula}]
Proposition \ref{prop:geometric-formula-generic} together with the fact that $V^\Phi_J$ are polynomials (Remark \ref{rmk: volumes are polynomials}) imply that $|\leq\theta(\lambda)|$ is a polynomial of degree $n$ in the $m_1,\ldots,m_n$ when they are positive integers. 
By Proposition \ref{prop:quasi-polynomiality} we know that $|\leq\theta(\lambda)|$ is a quasi-polynomial in the $m_i$'s.
If two quasipolynomials agree on the set of positive integers then they must agree everywhere, by Lemma \ref{lem:aux}.
Hence formula \eqref{eq-final} holds for every orbit polytope, with $\lambda$ generic or not.

Finally by Corollary \ref{cor:independent} the volume polynomials are linearly independent hence the coefficients $\mu_J^{\Phi}$ are unique.
\end{proof}

This implies that if $\Phi$ has rank $n$ and ${\lambda=(m_i)_{i\in I_n}}$ in the fundamental coweight basis, then ${|\leq\theta(\lambda)|}$ is a polynomial of degree $n$ in the ${m_1,\ldots,m_n}$. Taking the sum over a fixed rank ${|J|=d}$ gives the degree $d$ part of the polynomial. We call the coefficients $\mu_J^\Phi$ the \textit{geometric coefficients}.

\section{On the geometric coefficients \texorpdfstring{$\mu_J^\Phi$}{}}\label{mu}

In this section we compute some geometric coefficients $\mu_J^\Phi$. For the rest of this section, $\Phi$ will be a root system of rank $n$. 
We recall from Definition \ref{def:lattice} that  $\det(\Gamma)$ denotes the volume of the fundamental parallelepiped spanned by any basis of $\Gamma$.

\subsection{The extreme geometric coefficients \texorpdfstring{$\mu_\emptyset^\Phi$}{} and \texorpdfstring{$ \mu_{I_n}^\Phi$}{}} The geometric coefficient corresponding to the empty set is easily determined. Using the geometric formula \eqref{geofordef}, we get
\begin{equation*}
    \mu_\emptyset^\Phi=\sum_{J\subseteq I_n}\mu_J^\Phi V_J^\Phi(\mathbf{0})=|\leq\theta(\mathbf{0})|
    =|\leq w_0|
    =|W_f|.
\end{equation*}

\begin{lemma}\label{lem: biggeocoef}
    Let $\text{Vol}(A_\text{id})$ be the $n$-dimensional volume of the fundamental alcove. Then
\begin{equation*}
    \mu_{I_n}^\Phi=\dfrac{1}{\text{Vol}(A_\text{id})}.
\end{equation*}
\end{lemma}
\begin{proof} 

The Berline-Vergne construction $\nu$ has the property that $\nu$ for the whole polytope as a face is equal to 1.
Following equations \eqref{eq:geometric_coefficient} and \eqref{eq:v_J}, we have that
\begin{equation}\label{eq: coeftodo}
    \mu_{I_n}^\Phi = \frac{|W_f|}{\det(\mathbb{Z}\Phi^\vee)}.
\end{equation}
On the other hand, by \cite[Theorem 10.9]{barvinok2008integer}, we know that $$[\Lambda^\vee:\mathbb{Z}\Phi^\vee]=\det(\mathbb{Z}\Phi^\vee)/\det(\Lambda^\vee).$$
Using \eqref{eq: wforder} and substituting in \eqref{eq: coeftodo}, we get
\begin{equation} \label{geo coef In}
    \mu_{I_n}^\Phi = \dfrac{n!\eta_1\cdots\eta_n}{\det(\Lambda^\vee)}= \dfrac{1}{\vol(\FA)},
\end{equation}
where the last equality follows from Equation \eqref{eq:alcove_volume}.
\end{proof}

By \eqref{geo coef In} in order to compute the value of $\mu_{I_n}^{\Phi}$ we need to compute both the product $\eta_{1}\cdots \eta_{n}$ and $\det (\Lambda)^\vee$.  The values of $\mu_{I_n}^{\Phi}$  are computed using \cite[Plates I,\ldots,VI]{Bour46}  and they are  displayed in Table \ref{tab:my_label}. 

\begin{table}[ht]
\footnotesize
    \centering
\begin{tabular}{||c | c| c| c | c| c| c | c| c| c  ||} 
 \hline
 \mbox{Type} & $A_n$ & $B_n$ & $C_n$ & $D_n$  &  $E_6$  &  $E_7$ & $E_{8}$ & $F_4$ & $G_2$ \\ [0.5ex] 
 \hline\hline
$\eta_1\cdots \eta_n $ & $1 $& $2^{n-1}$ & $2^{n-1}$  & $2^{n-3}$  & $24$ & $288$ & $17280$ & $48$ & $6$\\
\hline
$\det (\Lambda^\vee)$ &  $\displaystyle \frac{\sqrt{n+1}}{n+1}$  & 1 & $\displaystyle \frac{1}{2}$ & 2 & $\displaystyle \frac{1}{\sqrt{3}}$ &$\displaystyle \frac{1}{\sqrt{2}}$  & 1 &  $2$  &  $\displaystyle \frac{1}{\sqrt{3}}$ \\ [2ex] 
\hline
$\mu_{I_n}^{\Phi}$ & $\displaystyle\frac{(n+1)!}{\sqrt{n+1}}$ & $n!2^{n-1}$ &$n!2^{n} $&$n!2^{n-4}$  & $ 24\sqrt{3} \cdot 6!  $ & $288\sqrt{2} \cdot 7!$ & $17280\cdot 8! $ & $576$ & $12\sqrt{3}$  \\
\hline
\end{tabular}
    \caption{Values of the geometric coefficient $\mu_{I_n}^{\Phi}$.}
    \label{tab:my_label}
    \normalsize
\end{table}

\subsection{Type A} In this section we fix a root system $\Phi$ of type $A_n$. We are going to compute  all the geometric coefficients $\mu_J^{A_n}$ for connected $\emptyset\neq J\subseteq I_n$.

For two positive integers $d,k$ with $k\leq d$, let $\Delta_{k,d}$ be the hypersimplex. In formulas, 
\begin{equation*}
    \Delta_{k,d}=\left\{ x\in [0,1]^d \mid x_1+\cdots+x_d=k\right\}.
\end{equation*}
The vertices of this convex polytope lie in $\mathbb{Z}^d$. Indeed, $\Delta_{k,d}$ is the convex hull of the vectors  whose coordinates consist of $k$ ones and $d-k$ zeros. Equivalently, $\Delta_{k,d}$ is the convex hull of the $S_{d}$-orbit of the vector $(1,\ldots ,1,0,\ldots , 0) \in \mathbb{R}^{d}$ with $k$ ones, where $S_{d}$ acts by permuting coordinates. 

The \emph{Ehrhart polynomial} of $\Delta_{k,d}$ is the polynomial $E_{k,d}(t)\in\bbQ[t]$ such that for every $m\in\mathbb{Z}_{\geq0}$,
\begin{equation*}
    E_{k,d}(m)= \left|\mathbb{Z}^d\cap m\Delta_{k,d} \right|,
\end{equation*}
where $m\Delta_{k,d}$ is the dilation of $\Delta_{k,d}$ with respect to the origin by the factor $m$.

\begin{lemma}\label{lem: simplex} For all $m\in\mathbb{Z}_{\geq0}$, and for all $1\leq k\leq n$,
\begin{equation}
    |\leq\theta(m\varpi_k)|=(n+1)! \, E_{k,n+1}(m).
\end{equation}
\end{lemma}

\begin{proof}
Recall the realization of the root system of type $A_n$ inside the subspace $V$ of  $\mathbb{R}^{n+1}$ of vectors whose coordinate sum is $0$, explained in  \S\ref{ex:typeA}. By  Theorem \ref{thm:first-lattice-formula}, we just need to prove the equality 
\begin{equation}\label{Tmaps}
      \vert \mathsf{P}^\Phi(m\varpi_k) \cap (m\varpi_{k} +  \mathbb{Z}\Phi)\vert=\vert\mathbb{Z}^{n+1}\cap m\Delta_{k,n+1}\vert. 
\end{equation}

The dilated hypersimplex $m\Delta_{k,n+1}$ does not belong to the linear subspace $V\subset \mathbb{R}^{n+1}$ but to the parallel affine hyperplane  $V_{mk}$ of 
$\mathbb{R}^{n+1}$ of vectors whose coordinate sum is $mk$.  

We will prove that the map $T:V_{mk}\rightarrow V$ given by 
\begin{equation}
    T(u) = u-\left(\dfrac{mk}{n+1} \right) \sum_{j=1}^{n+1} \varepsilon_j
\end{equation}
is a bijection that restricted to $(m\Delta_{k,n+1} \cap \mathbb{Z}^{n+1} )$ gives exactly the set $$\mathsf{P}^\Phi(m\varpi_k) \cap (m\varpi_{k} +  \mathbb{Z}\Phi),$$ thus proving Equation \eqref{Tmaps}. As $T$ is a translation it is a bijection. The  set $m\Delta_{k,n+1}$ is the convex hull of the $W_f=S_{n+1}$  orbit of the vector $v_k\coloneqq m\varepsilon_1+\ldots + m\varepsilon_k$. Using \eqref{eq:fundamenta_coweight} one sees that
 $T(v_k) = m\varpi_{k}$. Since both $m\Delta_{k,n+1}$ and $\mathsf{P}^\Phi(m\varpi_k)$ are defined as the convex hull of the $W_f=S_{n+1}$  orbit of the vectors $v_k $ and $m\varpi_{k}$, respectively, and $T$ clearly commutes with the action of $S_{n+1}$,
 we conclude that  $T(m\Delta_{k,n+1}) = \mathsf{P}^\Phi(m\varpi_k)$. 

 On the other hand, for $u=(u_1,\ldots , u_{n+1}) \in m\Delta_{k,n+1} \cap \mathbb{Z}^{n+1} $ we have
 $$ T(u)-m\varpi_k = u -v_k   = \sum_{j=1}^n  \left( \sum_{i=1}^j u_i - \left(  \min\{ j,k \}\right) m \right) \alpha_j  \in \mathbb{Z}\Phi.   $$
This shows that 
 \begin{equation} \label{eq dilated 1}
     T (m\Delta_{k,n+1} \cap \mathbb{Z}^{n+1} ) \subseteq \mathsf{P}^\Phi(m\varpi_k) \cap (m\varpi_{k} +  \mathbb{Z}\Phi).   
 \end{equation}
Conversely, for  $u=(u_1,\ldots , u_{n+1}) \in m\varpi_k +\mathbb{Z}\Phi $ we consider $u'=u- m\varpi_k + v_k $. It is easy to see that $T(u')=u$ and that 
$  u'\in v_k +  \mathbb{Z}\Phi \subset \mathbb{Z}^{n+1}$.
This shows that 
\begin{equation} \label{eq dilated 2}
  \mathsf{P}^\Phi(m\varpi_k) \cap (m\varpi_{k} +  \mathbb{Z}\Phi) \subseteq  T (m\Delta_{k,n+1} \cap \mathbb{Z}^{n+1} ).   
\end{equation}
Finally, by combining \eqref{eq dilated 1} and \eqref{eq dilated 2} we get \eqref{Tmaps}. 
\end{proof}




In \cite[Lemma 4.1]{Ferroni} the author provided closed formulas for the coefficients of the polynomial $E_{k,d}(t)$. For $0\leq m<d$, the coefficient $\tilde{e}_{k,d,m}$ of $t^m$ in $E_{k,d}(t)$ is given by
\begin{equation}\label{eq: Ferroni coefs}
    \tilde{e}_{k,d,m}=\dfrac{1}{(d-1)!}\,\sum_{j=0}^{k-1}\sum_{i=0}^{d-m-1}(-1)^{i+j}\binom{d}{j}(k-j)^m\left[
\begin{matrix}
d-j \\
m+1+i-j \\
\end{matrix}
\right]\left[
\begin{matrix}
j \\
j-i \\
\end{matrix}
\right],
\end{equation}
where the brackets denote the (unsigned) Stirling numbers of the first kind \cite[A008275]{oeis}. 
Therefore, by Lemma \ref{lem: simplex}, we can write
\begin{equation}\label{eq: eeeee}
|\leq\theta(t\varpi_i)|=e_{i,0}+e_{i,1}\,t+\cdots+e_{i,n}\,t^n,
\end{equation}
with $e_{i,j}=(n+1)!\,\tilde{e}_{i,n+1,j}$.

On the other hand, for $a$ a non-negative integer, we can obtain another expression for $|\leq\theta(a\varpi_i)|$ by applying the  Geometric Formula \eqref{geofordef}.
Namely,  
\begin{equation}\label{eq: twi}
    |\leq\theta(a\varpi_i)|=\sum_{J\subseteq I_n}\mu_J^{A_n} V_J^{A_n}(a\varpi_i).
\end{equation}

We can compute $V_J^{A_n}(a\varpi_i)$ explicitly.

\begin{lemma}\label{lem: vol twi} Let $\emptyset\neq J\subseteq I_n$ and $a\in\bbZ_{\geq0}$. 
Then, 
$V_J^{A_n}(a\varpi_i) = c_{i,J}a^{|J|}$
where $c_{i,J}$ is the coefficient of $m_i^{|J|}$ in the polynomial $V_J^{A_n}(\mathbf{m})$. Moreover, $c_{i,J}=0$ unless $i\in J$ 
and $J$  is connected.
\end{lemma}
\begin{proof} 
The polynomial $V_J^{A_n}(\mathbf{m})$ is homogeneous of degree $|J|$ in the variables $(m_j)_{j\in J}$. It follows that  $V_J^{A_n}(a\varpi_i)=c_{i,J} \, a^{|J|}$.

By Lemma \ref{lem: volgralfacts}\eqref{VolFact1}  if $i\notin J$, we have $V_J^{A_n}(a\varpi_i)= V_J^{A_n}(\mathbf{0})=0$. Finally, let us assume that  $J$ is not connected. Let  $J_1$ be a connected component of $J$ such that $i\not \in J_1$. We saw above that $V_{J_1}^{A_n}(a\varpi_i)=0 $. On the other hand, Lemma \ref{lem: volgralfacts}\eqref{VolFact2} shows that  
 $V_{J_1}^{A_n}$ is a factor of $V_J^{A_n}$. We conclude that $V_{J}^{A_n}(a\varpi_i)=0 $. 
\end{proof}

We can give an explicit formula for the coefficients $c_{i,J}$ occurring in Lemma \ref{lem: vol twi}. 
Note that every non-empty connected set in type $A_n$ is given by $I(l,u)\coloneqq I_l+u$ for some $1\leq l\leq n$ and $0\leq u \leq n-l$.

\begin{lemma} \label{lem: c coeficientes}
   Let $1\leq l\leq n$ and $0\leq u \leq n-l$. If  $i\in I(l,u)$ then we have
\begin{equation*}
    c_{i,I(l,u)}=\dfrac{\sqrt{l+1}}{l!}A(l,i-u),
\end{equation*}
where $A(r,s)$ denotes the Eulerian number \cite[\text{A008292}]{oeis}.
\end{lemma}

\begin{proof}
 By  Lemma \ref{lem: generalized pyramid formula} we have $c_{i,I(l,u)} = c_{i-u,I_l}$.
 Therefore, it is enough to show that 
 \begin{equation}
     c_{j,I_l} = \dfrac{\sqrt{l+1}}{l!}A(l,j)
 \end{equation}
for all $1\leq j \leq l$. This last equality follows by  \cite[Theorem 16.3(3)]{postnikov2009permutohedra} via  Remark \ref{rmk: postinkov}.
\end{proof}

For $1\leq l,i\leq n$, let $\mathcal{C}_i^l$ be the collection of connected subsets $J\subseteq I_n$ such that $i\in J$ and $|J|=l$. 
Although the left-hand side of \eqref{eq: twi} is only defined for $a\in \mathbb{Z}_{\geq 0}$, the right-hand side is a polynomial in the variable $a$. We know that two polynomials that agree in an infinite number of points are equal,  
so by equating the coefficients in 
\eqref{eq: eeeee} and \eqref{eq: twi} we get 
\begin{equation}\label{eq: coefsyst}
\sum_{J\in\mathcal{C}_i^l}\mu_J^{A_n}\,c_{i,J}=e_{i,l}, 
\end{equation}
for all $1\leq l,i\leq n$. 
We can reformulate this problem as a family of systems of linear equations as follows. For a fixed $1\leq l \leq n$, let $M_l=(m_{ij}) \in M_{n-l+1}(\mathbb{R}) $ where
\begin{equation}
    m_{ij} = \left\{ \begin{array}{rl}
        c_{i,I(l,j-1)},  & \mbox{if } i \in I(l,j-1);  \\
         0,& \mbox{otherwise. } 
    \end{array}
    \right.
\end{equation}
Furthermore, let $\mu_l = \left( \mu_{I(l,0)}^{A_n},\mu_{I(l,1)}^{A_n}, \ldots , \mu_{I(l,n-l)}^{A_n} \right)^t$ and $F_l=(e_{1,l},e_{2,l},\ldots , e_{n-l+1,l})^t$, where $v^t$ denotes the transpose of $v$. Then, the system of linear equations $M_{l}\mu_l= F_l$ is equivalent to the subset of equations in \eqref{eq: coefsyst}, obtained by considering $1\leq i \leq n-l+1$.

It is easy to see that $M_l$ is a lower triangular matrix with determinant 
\begin{equation}
    \det (M_l) = \prod_{i=1}^{n-l+1}  c_{i,I(l,i-1)}.
\end{equation}
By Lemma \ref{lem: c coeficientes} we know that $c_{i,I(l,i-1)}\neq 0$ for all $1\leq i \leq n-l+1$. Therefore, $M_{l}$ is invertible and we can obtain the geometric coefficients for $J$ connected by solving the system $M_l\mu_l = F_l$.  

We finish this paper by providing closed formulas for geometric coefficients in some specific cases. For instance, if $l=1$ then $M_l$ is a diagonal matrix. Thus, we get
 \begin{equation}
     \mu_{\{i\}}^{A_n} =  \dfrac{e_{i,1}}{c_{i,\{i\}}}.
 \end{equation}
Similarly, as $M_l$ is a lower triangular matrix we can solve the equality associated with the first row of $M_l$ in the system $M_l\mu_l = F_l$ for all $1\leq l \leq n$. This yields
\begin{equation} 
     \mu_{I_l}^{A_n} =  \dfrac{e_{1,l}}{c_{1,I_l}} = \frac{l!}{\sqrt{l+1}}(n+1)\left[
\begin{matrix}
n+1 \\
l+1 \\
\end{matrix}
\right],   
\end{equation}
 which corresponds to \eqref{eq: intro geo coeff}. 

\begin{remark}\label{remarkfinal}
In \cite{castillo2021todd} the authors provided a combinatorial formula for the $\nu$ function (it is called $\alpha$ in the reference) from Berline and Vergne in the type A case.
By Equation \eqref{eq:geometric_coefficient}, the positiveness of the geometric coefficients $\mu_J^\Phi$ is directly tied to the positivity of $\nu$.
By \cite[Example 6.3]{castillo2021todd} there exists faces such that the $\nu$ function is negative. 
Hence there are also negative $\mu$ coefficients.
\end{remark}

\bibliography{bibliography}
\bibliographystyle{plain}
\end{document}